\documentclass[10pt,reqno]{amsart}
\usepackage{amsmath,amscd,amssymb}
\usepackage{color}
\usepackage{tikz}
\usepackage{hyperref}

\theoremstyle{plain}
   \newtheorem{theorem}{Theorem}[section]
   \newtheorem{proposition}[theorem]{Proposition}
   \newtheorem{lemma}[theorem]{Lemma}
   \newtheorem{corollary}[theorem]{Corollary}

\theoremstyle{definition}
   \newtheorem{definition}{Definition}[section]
   
   \newtheorem{question}{Question}[section]
   \newtheorem{example}{Example}[section] 
\theoremstyle{remark}
 \newtheorem{remark}{Remark}[section]

\newcommand{\R}{\mathbb{R}}
\newcommand{\Proj}{\mathbb{P}}
\newcommand{\Z}{\mathbb{Z}}

\newcommand{\I}{\mathcal{I}}
\newcommand{\SW}{\mathcal{SW}}
\newcommand{\X}{\mathcal{X}}

\def\newop#1{\expandafter\def\csname #1\endcsname{\mathop{\rm
#1}\nolimits}}

\newop{Ehr}
\newop{conv}
\newop{Des}
\newop{des}
\newop{Asc}
\newop{asc}
\newop{exc}
\newop{rev}
\newop{link}
\newop{supp}
\DeclareMathOperator{\verts}{vert}

\subjclass[2000]{}

\begin{document}
\title{Derangements, Ehrhart Theory, and Local $h$-polynomials}

\author{Nils Gustafsson}
\date{\today}
\address{Matematik, KTH, SE-100 44 Stockholm, Sweden}
\email{nilsgust@kth.se}

\author{Liam Solus}
\date{\today}
\address{Matematik, KTH, SE-100 44 Stockholm, Sweden}
\email{solus@kth.se}

\begin{abstract}
The Eulerian polynomials and derangement polynomials are two well-studied generating functions that frequently arise in combinatorics, algebra, and geometry. 
When one makes an appearance, the other often does so as well, and their corresponding generalizations are similarly linked.  
This is this case in the theory of subdivisions of simplicial complexes, where the Eulerian polynomial is an $h$-polynomial and the derangement polynomial is its local $h$-polynomial.  
Separately, in Ehrhart theory the Eulerian polynomials are generalized by the $h^\ast$-polynomials of $s$-lecture hall simplices.  
Here, we show that derangement polynomials are analogously generalized by the box polynomials, or local $h^\ast$-polynomials, of the $s$-lecture hall simplices, and that these polynomials are all real-rooted.  
We then connect the two theories by showing that the local $h$-polynomials of common subdivisions in algebra and topology are realized as local $h^\ast$-polynomials of $s$-lecture hall simplices.  
We use this connection to address some open questions on real-rootedness and unimodality of generating polynomials, some from each side of the story.
\end{abstract}


\maketitle
\thispagestyle{empty}

\section{Introduction}
\label{sec: introduction}
For a positive integer $d$, let $[d]:=\{1,\ldots,d\}$ and $[d]_0:=\{0,\ldots,d\}$.  
Given a discrete random variable $X: \X\longrightarrow[d]_0$ with finite sample space $\X$, let $p_k:=|X^{-1}(k)|$ for all $k\in[d]_0$.  
The polynomial 
\[
p(X; z) := p_0+p_1z+\cdots+p_dz^d,
\]
is called the {\bf generating polynomial} for $X$, and its coefficients encode the discrete probability distribution $\Proj[X =  k]$.   
A longstanding endeavor in combinatorics is to understand the properties of the distribution encoded by $p(X;z)$ when the sample space $\X$ and random variable $X$ are combinatorially significant.  
Researchers are often interested in deciding when a specific generating polynomial encodes a distribution with the statistically desirable features of the binomial distribution $B(d,1/2)$.  
Consequently, a polynomial $p(z) = p_0+p_1z+\cdots+p_dz^d\in\R[z]$ is called {\bf unimodal} if 
$
p_0\leq p_1\leq \cdots\leq p_t\geq \cdots \geq p_{d-1}\geq p_d,
$
for some index $t\in[d]$.  
It is called {\bf symmetric} (with respect to degree d) if $p_k = p_{d-k}$ for all $k\in[d]_0$, and it is called {\bf log-concave} if $p_k^2\geq p_{k-1}p_{k+1}$ for all $k\in[d-1]$.  
The inequalities defining these distributional properties are also useful in classification and equidistribution problems for generating polynomials.  
When proving distributional properties for a generating polynomial $p(X;z)$, we can make use of its factorizations, and specifically, its zeros.  
The polynomial $p(z)$ is called {\bf real-rooted} if it only has real zeros, or if $p(z)\equiv 0$.
If $p(z)$ is real-rooted with nonnegative coefficients then it is log-concave and unimodal \cite[Theorem 1.2.1]{B89}. 
Thus, it is particularly desirable if a generating polynomial of combinatorial significance is both symmetric and real-rooted.

Two fundamental, and closely related, symmetric and real-rooted generating polynomials arise when $\X=\mathfrak{S}_n$, the collection of all permutations of $[n]$.  
Given a permutation $\pi = \pi_1\pi_2\ldots\pi_n\in\mathfrak{S}_n$, we say that an index $i\in[n-1]$ is a {\bf descent} in $\pi$ if $\pi_i>\pi_{i+1}$, and we let $\des(\pi)$ denote the number of descents in $\pi$.  
We say that $i$ is an {\bf excedance} in $\pi$ if $\pi_i>i$, and similarly let $\exc(\pi)$ denote the number of excedances in $\pi$.  
We also say $\pi$ is a {\bf derangement} if $\pi_i\neq i$ for all $i$, and we let $\mathfrak{D}_n$ denote the collection of all derangements in $\mathfrak{S}_n$.  
The polynomials
\[
A_n(z) :=\sum_{\pi\in\mathfrak{S}_n}z^{\des(\pi)} 
\qquad
\mbox{and}
\qquad
d_n(z):=\sum_{\pi\in\mathfrak{D}_n}z^{\exc(\pi)}
\]
are, respectively, known as the {\bf $n^{th}$ Eulerian polynomial} and the {\bf $n^{th}$ derangement polynomial}.  
These two polynomials appear, and often together, in a wide variety of settings throughout combinatorics, geometry, and algebra.  
When they arise, researchers study how their distributional properties generalize within the given context.  
In this paper we will compare how $A_n(z)$ and $d_n(z)$ arise in the theory of local $h$-polynomials for subdivisions of simplicial complexes with how they appear in the theory of (Ehrhart) $h^\ast$-polynomials of simplices and their associated box polynomials \cite{BR07}, which are sometimes called their local $h^\ast$-polynomials \cite{KS16}.  

Within the theory of subdivisions of simplicial complexes, $A_n(z)$ arises as the $h$-polynomial of the barycentric subdivision of a simplex and $d_n(z)$ as its local $h$-polynomial \cite{S92}.  
Generalizing this example, a variety of $h$-polynomials and their local $h$-polynomials for well-studied subdivisions in algebra and topology relate closely to generalizations of $A_n(z)$ and $d_n(z)$ \cite{A16}.  
In Ehrhart theory, $A_n(z)$ is generalized by the {$s$-Eulerian polynomials}, which are the $h^\ast$-polynomials of the {$s$-lecture hall simplices} \cite{SS12}.  
We show here that the local $h^\ast$-polynomials of the $s$-lecture hall simplices analogously generalize $d_n(z)$.  
We call these polynomials the $s$-derangement polynomials, and show that they are both real-rooted and symmetric.

Local $h$-polynomials were introduced in \cite{S92} so as to create a parallel to a useful theorem concerning local $h^\ast$-polynomials and their $h^\ast$-polynomials \cite{BM85}.  
It is therefore natural to ask when these two theories intersect, and when this intersection can be used to answer questions in one context via methods from the other.  
One main contribution of this paper is to show that many of the well-studied local $h$-polynomials for subdivisions of simplices are realized as local $h^\ast$-polynomials of $s$-lecture hall simplices; i.e.,~as $s$-derangement polynomials.  
In doing so, we also show that the $s$-derangement polynomials generalize other well-studied derangement polynomials.  
We apply these results to answer some open questions, some from each of the two intersecting theories, on real-rootedness and unimodality of generating polynomials.

In Section~\ref{sec: preliminaries} we recall the details of local $h$-polynomials, local $h^\ast$-polynomials, their parallel story, and their applications.  
In Section~\ref{sec: s-derangement polynomials}, we prove that the local $h^\ast$-polynomials of the $s$-lecture hall simplices (the $s$-derangement polynomials) are always symmetric and real-rooted.  
In Section~\ref{sec: examples}, we show that the $s$-derangement polynomials generalize $d_n(z)$, as well as other derangement polynomials, and we show that the key examples of local $h$-polynomials can be realized as $s$-derangement polynomials.  
In Section~\ref{sec: applications}, we use these results to settle some conjectures on the real-rootedness of certain local $h$-polynomials and related generating polynomials.  
We then prove that the family of $s$-lecture hall order polytopes \cite{BL16}, which generalize the well-studied order polytopes \cite{S86}, admit a box unimodal triangulation \cite{SV13}.  
We recover from this that all reflexive $s$-lecture hall order polytopes have unimodal $h^\ast$-polynomials, and we use this fact to partially answer a conjecture posed in \cite{BL16}.

\section{Local $h$-polynomials and local $h^\ast$-polynomials}
\label{sec: preliminaries}
In this section, we summarize a portion of the theory of subdivisions and local $h$-polynomials, and the analogous story for local $h^\ast$-polynomials in Ehrhart theory.  
The connections outlined in this section motivate results in the later sections of the paper where we explicitly relate families of local $h$-polynomials to local $h^\ast$-polynomials.  

\subsection{Subdivisions and local $h$-polynomials}
\label{subsec: subdivisions and local h-polynomials}
Let $\Omega$ be a $(n-1)$-dimensional (abstract) simplicial complex, let $f_i$ denote the number of faces of $\Omega$ of dimension $i$ for $i\in[n-1]_0$ and $f_{-1}:=1$ when $\Omega\neq\emptyset$.
The {\bf $f$-polynomial} of $\Omega$ is defined as
\[
f(\Omega;z) := \sum_{i=0}^{n}f_{i-1}z^{i}.
\]
In algebraic combinatorics, researchers sometimes find it easier to work with the {\bf $h$-polynomial} of $\Omega$, which is defined as
\[
h(\Omega;z)  = \sum_{i=0}^nh_iz^i := (1-z)^{n}f\left(\Omega\,;\frac{z}{1-z}\right).
\]
The $h$-polynomial is known to have only nonnegative coefficients when $\Omega$ is {\bf Cohen-Macaulay} \cite{S07}; for example, when $\Omega$ is a homology ball or sphere.
It is also symmetric with respect to degree $n$ whenever $\Omega$ is a homology sphere. 
For a complete discussion of this topic, and for all unknown definitions, we refer the reader to \cite{S07}.  

A {\bf topological subdivision} of $\Omega$ is a simplicial complex $\Omega^\prime$ such that each simplex $\Delta\in\Omega$ is subdivided into a ball by simplices in $\Omega^\prime$ so that the boundary this ball is a subdivision of the boundary of $\Delta$.  
It is further called {\bf geometric} if both $\Omega$ and $\Omega^\prime$ admit geometric realizations, $\Sigma$ and $\Sigma^\prime$, respectively, (i.e.~with each simplex realized as a convex simplex in real-Euclidean space) such that the vertices of $\Sigma$ are a subset of the vertices of $\Sigma^\prime$ and each face of $\Sigma^\prime$ is contained in a face of $\Sigma$. 
In between geometric and topological subdivisions we also have {\bf quasi-geometric subdivisions}.
These are the topological subdivisions $\Omega^\prime$ of $\Omega$ such that no simplex in $\Omega^\prime$ has all of its vertices in a face of smaller dimension in $\Omega$. 
Given a subdivision $\Omega^\prime$ of $\Omega$, we may often refer to the associated inclusion map $\varphi: \Omega^\prime\longrightarrow\Omega$.  
We refer the reader to \cite{A16} for more details on the various types of subdivisions.  

For polynomials $p(z) = p_0+p_1z+\cdots+p_dz^d$ and $q(z) = q_0+q_1z+\cdots+q_mz^m$, we write $p(z)\leq q(z)$ if $p_i\leq q_i$ for all $i\geq 0$.  
The statement $p(z)\leq q(z)$ is referred to as a {\bf monotonicity property}.
One natural question, asked by Kalai and Stanley, is how does the $h$-polynomial of a Cohen-Macaulay simplicial complex change when it is subdivided by another simplicial complex? 
In particular, Kalai and Stanley asked the following: if $\Omega^\prime$ is a subdivision of a Cohen-Macaulay simplicial complex $\Omega$, does it follow that $h(\Omega;z)\leq h(\Omega^\prime;z)$?  

To affirmatively answer this question in the case of quasi-geometric subdivisions of Cohen-Macaulay complexes, Stanley introduced the notion of a local $h$-polynomial in \cite{S92}.  
Let $2^V$ denote the abstract $(n-1)$-dimensional simplicial complex consisting of all subsets of a set $V$ with $|V| = n$, and let $\Omega$ be a subdivision of $2^V$ with associated map $\varphi:\Omega\longrightarrow2^V$.  
For a face $F\in 2^V$, we let $\Omega_F:=\varphi^{-1}(2^F)$.
The {\bf local $h$-polynomial} $\ell_V(\Omega_V;z)$ of $2^V$ is then defined by the relation
\begin{equation}
\label{eqn: local}
h(\Omega;z) = \sum_{F\subset V}\ell_F(\Omega_F;z).
\end{equation}
It is important to note that since equation~\eqref{eqn: local} holds for all subdivisions of all simplices then the Principle of Inclusion-Exclusion implies that
\[
\ell_V(\Omega_V;z) = \sum_{F\subset V}(-1)^{n-|F|}h(\Omega_F;z),
\]
and hence equation~\eqref{eqn: local} does in fact determine the local invariant $\ell_V(\Omega_V;z)$.  

If $\Delta\in \Omega$, we define the {\bf link} of $\Delta$ in $\Omega$ to be the collection of all simplices in $\Omega$ that are disjoint from $\Delta$ but are contained in a face of $\Omega$ that also contains $\Delta$; i.e., 
\[
\link_\Omega(\Delta) :=
\{
\sigma\in \Omega 
\, : \,
\sigma\cap\Delta = \emptyset, 
\mbox{ and there exists $\Delta^\prime\in \Omega$ such that $\Delta,\sigma\subset\Delta^\prime$}
\}.
\]
To prove the desired monotonicity property, Stanley showed that $\ell_V(\Omega_V;z)$ has only nonnegative integer coefficients when $\Omega$ is a quasi-geometric subdivision of $2^V$, and he proved the following theorem.  
\begin{theorem}
\label{thm: local decomposition}
\cite[Theorem 3.2]{S92}
Let $\Omega$ be a pure $(n-1)$-dimensional simplicial complex and let $\Omega^\prime$ be a simplicial subdivision of $\Omega$.  Then
\[
h(\Omega^\prime;z) = \sum_{\Delta\in\Omega}h(\link_\Omega(\Delta);z)\ell_\Delta(\Omega_\Delta^\prime;z).
\]
\end{theorem}

Since the coefficients of $\ell_V(\Omega_V;z)$ are nonnegative integers whenever $\Omega$ is a quasi-geometric subdivision, researchers have since investigated combinatorial interpretations of these coefficients for common subdivisions used in topology and algebra.  
For example, if $\Omega$ is the {\bf barycentric subdivision} of the $(n-1)$-simplex $2^V$ where $|V| = n$, then it turns out that
$
h(\Omega;z) = A_n(z), 
$
and 
$
\ell_V(\Omega_V;z) = d_n(z)
$
\cite{S92}.
In the next subsection, we see that the definition of a local $h$-polynomial finds its motivation in an analogous construction for $h^\ast$-polyomials of lattice polytopes, called local $h^\ast$-polynomials.  

\subsection{Some Ehrhart theory}
\label{subsec: some ehrhart theory}
Let $P\subset\R^n$ be a $d$-dimensional convex lattice polytope; i.e.~a convex polytope all of whose vertices lie in the lattice $\Z^n$ such that the points within $P$ span a $d$-dimensional affine subspace in $\R^n$.  
Throughout the paper, we will let $\verts(P)$ denote the set of vertices of the polytope $P$.
For $t\in\Z_{\geq0}$ we call $tP:=\{tp:p\in P\}$ the {\bf $t^{th}$ dilate} of $P$, and we call the generating function
\[
\Ehr_P(z) :=\sum_{t\geq0}|tP\cap\Z^n|z^t = \frac{h_0^\ast+h_1^\ast z+\cdots+h_d^\ast z^d}{(1-z)^{d+1}}
\]
the {\bf Ehrhart Series} of $P$.  
The polynomial $h^\ast(P;z) = h_0^\ast+h_1^\ast z+\cdots+h_d^\ast z^d$ is the {\bf (Ehrhart) $h^\ast$-polynomial} of $P$, and in \cite{S80}, it was shown that $h_0^\ast,\ldots,h_d^\ast\in\Z_{\geq0}$.  
Since the coefficients of $h^\ast(P;z)$ are all nonnegative integers, it is currently popular to investigate various combinatorial interpretations of these coefficients in terms of the polytope $P$, and then study their associated distributional properties.  
In the case that the lattice polytope $P$ is a $d$-simplex, say 
\[ 
P = \conv(v^{(0)},\ldots,v^{(d)})\subset \R^n,
\]
for $d+1$ affinely independent points $v^{(0)},\ldots,v^{(d)}$, then $h^\ast(P;z)$ has a well-known combinatorial interpretation in terms of the lattice points in the {\bf half-open parallelpiped} of $P$.  
This is the convex body 
\[
\Pi_P :=
\left\{
\sum_{i=0}^d\lambda_i(v_i,1)\in\R^{n+1} 
\,:\,
0\leq \lambda_i< 1, \, i\in[d]_0
\right\}.
\]
Notice here, that we embedded $P$ into $\R^{n+1}$ within the hyperplane defined by $x_{n+1} = 1$.  
We then have the following well-known combinatorial interpretation of $h^\ast(P;z)$, a proof of which can be found in \cite[Chapter 3]{BR07}.  
\begin{lemma}
\cite[Corollary 3.11]{BR07}
Let $P$ be a lattice $d$-simplex with half-open parallelpiped $\Pi_P$.  
Then 
\[
h^\ast(P;z) = \sum_{(x_1,\ldots,x_{n+1})\in\Pi_P\cap\Z^{n+1}}z^{x_{n+1}}.
\]
\end{lemma}

When $P$ is not a simplex, identifying useful interpretations of the coefficients of its $h^\ast$-polynomial becomes more challenging.  
One way to better understand the structure of the $h^\ast$-polynomial for an arbitrary polytope is to decompose it via a triangulation.  
We call a geometric realization $T$ of a simplicial complex a {\bf (lattice) triangulation} of a subset $S\subset\R^n$ if the union over all faces of $T$ is $S$ and all $0$-dimensional faces of $T$ are in $\Z^n$.  
Given a triangulation $T$ of a lattice polytope $P\subset\R^n$, our goal is then to compute $h^\ast(P;z)$ by counting lattice points in the relative interiors of the dilates of each simplex in $T$.  
To this end, for a lattice $d$-simplex $\Delta\subset\R^n$ we define the {\bf open-parallelpiped}
\[
\Pi_P^\circ :=
\left\{
\sum_{i=0}^d\lambda_i(v_i,1)\in\R^{n+1} 
\,:\,
0< \lambda_i< 1, \, i\in[d]_0
\right\},
\]
and the polynomial 
\[
\ell^\ast(\Delta;z) :=\sum_{(x_1,\ldots,x_{n+1})\in\Pi_P\cap\Z^{n+1}}z^{x_{n+1}}.
\]
The polynomial $\ell^\ast(\Delta;z)$ is called the {\bf local $h^\ast$-polynomial} of $\Delta$ in \cite{KS16} and the {\bf box polynomial} of $\Delta$ in \cite{BR07,B16,SV13}.  
It allows us to decompose the $h^\ast$-polynomial of a polytope as in the following theorem, from which we also see its close relation to the local $h$-polynomial of a simplex.  
\begin{theorem}
\label{thm: betke and mcmullen}
\cite{BM85}
Let $P$ be a lattice polytope with lattice triangulation $T$.  Then
\[
h^\ast(P;z) = \sum_{\Delta\in T}h(\link_T(\Delta);z)\ell^\ast(\Delta;z).
\]
\end{theorem}
Since the polynomials $h(\link_T(\Delta);z)$ and $\ell^\ast(\Delta;z)$ will always have nonnegative coefficients, then one can use Theorem~\ref{thm: betke and mcmullen} to deduce the monotonicity property $h^\ast(Q;z)\leq h^\ast(P;z)$ whenever $Q\subset P$.  
In this way, we see that Theorem~\ref{thm: local decomposition} was a natural approach for Stanley to prove the same monotonicity property for $h$-polynomials in the setting of quasi-geometric subdivisions of Cohen-Macaulay simplicial complexes \cite{S92}.  

Given the strong similarities between Theorems~\ref{thm: local decomposition} and~\ref{thm: betke and mcmullen}, it is natural to ask how local $h$-polynomials and local $h^\ast$-polynomials can relate as combinatorial generating functions, especially when their $h$- and $h^\ast$-polynomials are enumerating similar combinatorial objects.  
In particular, we consider the following question:
\begin{question}
\label{quest: local relations}
Suppose the $h$-polynomials of a family of topological subdivisions $\mathcal{F}$ are enumerating similar combinatorial objects to the $h^\ast$-polynomials of a family of lattice simplices $\mathcal{S}$.  
Given a local $h$-polynomial $\ell_V(\Omega_V;z)$ for $\Omega\in\mathcal{F}$ can we find a lattice simplex $S\in\mathcal{S}$ such that
$
\ell^\ast(S;z) = \ell_V(\Omega_V;z)?
$
\end{question}
One main contribution of this paper is to affirmatively answer Question~\ref{quest: local relations} when the $h$- and $h^\ast$-polynomials are enumerating descents statistics for inversion sequences, a family of combinatorial objects generalizing $\mathfrak{S}_n$.
As we will see in Section~\ref{sec: applications}, a well-chosen answer to Question~\ref{quest: local relations} can help us settle conjectures about the given local $h$-polynomial.

More recently, there has also been a strong interest as to when Theorem~\ref{thm: betke and mcmullen} can be used to prove unimodality for $h^\ast$-polynomials \cite{B16,SV13}; specifically in the case of reflexive polytopes.  
A $d$-dimensional lattice polytope $P$ is {\bf reflexive} (up to unimodular transformation) if and only if $h^\ast(P;z)$ is symmetric with respect to degree $d$ \cite{H92}.  
For a reflexive polytope $P$, Theorem~\ref{thm: betke and mcmullen} reduces to a statement about a triangulation of the boundary of $P$, which we denote by $\partial P$.  
\begin{theorem}
\label{thm: reflexive}
\cite{BM85}
Let $P$ be a reflexive lattice polytope, and let $T$ be a triangulation of its boundary, $\partial P$.  Then
\[
h^\ast(P;z) = \sum_{\Delta\in T}h(\link_T(\Delta);z)\ell^\ast(\Delta;z).
\]
\end{theorem}
In the case when the lattice triangulation $T$ of $\partial P$ is a {\bf regular triangulation} \cite[Definition 2.2.10]{DRS10}, then the polynomials $h(\link_T(\Delta);z)$ are all symmetric and unimodal of degree $d-\dim(\Delta)-1$ (see for instance \cite[Lemma 2.9]{S09}).  
It is an exercise to check that $\ell^\ast(\Delta;z)$ is always symmetric with respect to degree $\dim(\Delta)+1$. 
Thus, if a reflexive polytope $P$ admits a boundary triangulation whose simplices all have unimodal box polynomials then $h^\ast(P;z)$ is also unimodal.
In \cite{SV13}, such a triangulation is referred as a box unimodal triangulation.
\begin{definition}
\label{def: box unimodal}
A lattice triangulation of a subset $S\subset\R^n$ is called {\bf box unimodal} if it is regular and every $\Delta\in T$ has a unimodal box polynomial.
\end{definition}
In relation to unimodality conjectures for $h^\ast$-polynomials for reflexive polytopes, \cite{B16,SV13} asked which families of well-studied lattice polytopes do (or do not) admit a box unimodal triangulation.
A second main contribution of this paper is to prove that the well-studied family of $s$-lecture hall simplices \cite{BBKSZa,BBKSZb,BL16,KO17,S16,SS12,SV15} all have real-rooted, and thus unimodal, local $h^\ast$-polynomials (see Section~\ref{sec: s-derangement polynomials}).  
We then apply this result to prove that a family of lattice polytopes simultaneously generalizing $s$-lecture hall simplices and order polytopes \cite{S86}, called the $s$-lecture hall order polytopes \cite{BL16}, admit a box unimodal triangulation.  
As a consequence we recover unimodality results that allow us to provide a partial answer to a conjecture posed in \cite{BL16} (see Section~\ref{sec: applications}).  

\section{s-Derangement Polynomials}
\label{sec: s-derangement polynomials}
In this section, we prove that the local $h^\ast$-polynomials of the $s$-lecture hall simplices are all real-rooted, analogous to their $h^\ast$-polynomials.
Given a sequence of positive integers $s = (s_1,\ldots,s_n)$, we define the set of {\bf $s$-inversion sequences} of length $n$ to be 
\[
I_n^s :=
\{
(e_1,\ldots,e_n) \in\Z^n
\,:\,
0\leq e_i<s_i, \, 0\leq i\leq n
\}.
\]
We also let 
\[
\widetilde{s} := (s_0:=1,s_1,\ldots,s_n,s_{n+1} := 1),
\]
and note that $I_{n+2}^{\widetilde{s}}$ is isomorphic to $I_n^s$ via projection along the first and last coordinate.  
Thus, given an inversion sequence $(e_1,\ldots, e_n)\in I_n^s$, we may naturally assume that $e_0 = e_{n+1} =0$ and $s_0 = s_{n+1} = 1$. 

For an inversion sequence $e = (e_1,\ldots, e_n)\in I_n^s$, we say that an index $i \in[n]_0$ is an {\bf ascent} if $\frac{e_i}{s_i}<\frac{e_{i+1}}{s_{i+1}}$, and we denote the number of ascents in $e$ by $\asc(e)$. 
Similarly, we say that an index $i \in[n]_0$ is a {\bf descent} if $\frac{e_i}{s_i}>\frac{e_{i+1}}{s_{i+1}}$, and we denote the number of descents in $e$ by $\des(e)$. 
The {\bf $s$-Eulerian polynomial} is defined as
\[
E_n^s(z) :=\sum_{e\in I_n^s}z^{\asc(e)}.  
\]
One important feature of $E_n^s(z)$ that we will use in the following sections is that ascents and descents are equidistributed over $I_n^s$.  
\begin{theorem}
\label{thm: equidistributed}
For any sequence of positive integers $s = (s_1,\ldots, s_n)$ we have that
\[
E_n^s(z) = \sum_{e\in I_n^s}z^{\asc(e)} = \sum_{e\in I_n^s}z^{\des(e)}.
\]
\end{theorem}

\begin{proof}
Let $f : I_n^s \rightarrow I_n^s$ be the involution defined by 
$
f(e)_i = -e_i \text{ mod } s_i.
$
We will now show that $\asc(e) = \des(f(e))$ and $\des(e) = \asc(f(e))$, which proves the theorem. 
Let $a_e$ be the number of ascents $i\in\asc(e)$ such that $e_i > 0$ and $e_{i+1} > 0$, and let $\alpha_e$ be the number of $i\in\asc(e)$ such that $e_i = 0$ and $e_{i+1} > 0$. 
Similarly, we let $d_e$ denote the number of $i\in\des(e)$ with $e_i>0$ and $e_{i+1}>0$, and $\delta_e$ denote the number of $i\in\des(e)$ with $e_i>0$ and $e_{i+1}=0$.  
Note that if $e_i > 0$, then $$\frac{f(e)_i}{s_i} = 1-\frac{e_i}{s_i}.$$
From this it follows that $a_e = d_{f(e)}$ and $d_e = a_{f(e)}$. 
For the ascents and descents counted by $\alpha_e$ and $\delta_e$, respectively, note that since $e_0 = e_{n+1} = 0$, both $\alpha_e$ and $\delta_e$ are equal to the number of contiguous segments of non-zero elements in $(0,e,0)$. 
In particular, $\alpha_e = \delta_e$. 
Also, since $f(e)_i = 0$ if and only if $e_i = 0$, it follows that 
$$\alpha_e = \delta_e = \alpha_{f(e)} = \delta_{f(e)}.$$
Now we can conclude that
$$
\des(e) = d_e + \delta_e = a_{f(e)} + \alpha_{f(e)} = \asc(f(e)),
$$
and
$$
\asc(e) = a_e + \alpha_e = d_{f(e)} + \delta_{f(e)} = \des(f(e)),
$$
which completes the proof.
\end{proof}

In the case when $s = (1,2,\ldots,n)$, we have that $E_n^s(z) = A_n(z)$, the $n^{th}$ Eulerian polynomial.  
In \cite{SV15} the authors showed that, analogous to $A_n(z)$, the $s$-Eulerian polynomials are always real-rooted.
\begin{theorem}
\label{thm: s-eulerian polynomials}
\cite[Theorem 1.1]{SV15}
The $s$-Eulerian polynomial $E_n^s(z)$ is real-rooted for any sequence of positive integers $s$. 
\end{theorem}
It turns out that the $s$-Eulerian polynomial $E_n^s(x)$ is actually the $h^\ast$-polynomial of a lattice $n$-simplex \cite{SS12}.  
For a sequence of positive integers $s = (s_1,\ldots, s_n)$ the {\bf $s$-lecture hall simplex} is 
\[
P_n^s := \left\{
(x_1,\ldots,x_n)\in\R^n
\,:\,
0\leq \frac{x_1}{s_1}\leq \frac{x_2}{s_2}\leq\cdots\leq\frac{x_n}{s_n}\leq 1
\right\}.
\]
In \cite{SS12} it was shown that $h^\ast(P_n^s;z) = E_n^s(z)$.
Given this geometric interpretation, it is natural to ask what can be said about the local $h^\ast$-polynomial of $P_n^s$.   
Analogous to the generalization of $A_n(z)$ to $E_n^s(z)$, we shall see that $\ell^\ast(P_n^s;z)$ generalizes the derangement polynomials $d_n(z)$, and that these polynomials are also all real-rooted.  
To see this, we must first derive a nice formula for $\ell^\ast(P_n^s;z)$.  To this end, we let
\[
\widetilde{I_n^s} := 
\left\{
(e_0,e_1,\ldots,e_{n+1})\in I_{n+2}^{\widetilde{s}} 
\, : \,
\frac{e_i}{s_i}\neq\frac{e_{i+1}}{s_{i+1}}
\mbox{ for all $i\in[n]_0$}
\right\}.
\]
Notice that $\widetilde{I_n^s}$ is isomorphic to a subset of $I_n^s$ via the same map taking $I_{n+2}^{\widetilde{s}}$ to $I_n^s$.  
This subset is precisely the subset of inversion sequences whose descents are enumerated by the local $h^\ast$-polynomial $\ell^\ast(P_n^s;z)$.  
\begin{proposition}
\label{prop: formula for box polynomial}
Let $s = (s_1,\ldots,s_n)$ be a sequence of positive integers.  Then the local $h^\ast$-polynomial of the $s$-lecture hall simplex $P_n^s$ is
\[
\ell^\ast(P_n^s;z) = \sum_{e\in\widetilde{I_n^s}}z^{\asc(e)} = \sum_{e\in\widetilde{I_n^s}}z^{\des(e)}.
\]
\end{proposition}

\begin{proof}
Let $\Pi_s$ denote the half-open parallelepiped of $P^s_n$, and $\Pi_s^\circ$ denote its open parallelpiped. 
Recall that, by definition,  $\ell^\ast(P_n^s;z)$ enumerates the ``heights'' $x_{n+1}$ of all $(x_1,\ldots,x_{n+1})\in\Pi_s\cap\Z^{n+1}$. 
Therefore, we would like to find a bijection between $\Pi^{\circ}_s \cap \mathbb{Z}^{n+1}$ and $\widetilde{I_n^s}$ that maps elements with last coordinate $k$ to elements with $k$ descents (or ascents). 
Let REM$: \Pi_s\cap\mathbb{Z}^{n+1} \rightarrow I^s_n$ be the map defined by $$\text{REM}(x)_i = x_i \text{ mod }s_i.$$
In \cite{LS14}, it was proven that this is a bijection such that $x_{n+1} = \des(\text{REM}(x))$. 
Therefore, it suffices to prove that REM$(\Pi_s^{\circ}\cap\mathbb{Z}^{n+1}) = \widetilde{I^s_n}$, and that ascents and descents are equidistributed on $\widetilde{I^s_n}$. 
It can be shown (see Lemma 3.5 in \cite{LS14}) that $\Pi_s$ can be written as the set of points in $\mathbb{R}^{n+1}$ satisfying
\begin{enumerate}
\item $0 \leq \frac{x_1}{s_1} < 1 ,$
\item $0 \leq \frac{x_{i+1}}{s_{i+1}} - \frac{x_i}{s_i} < 1$ for $i \in[n-1]$, and
\item $0 \leq x_{n+1} - \frac{x_n}{s_n} < 1.$
\end{enumerate}
Since $\Pi_s^{\circ}$ is the interior of $\Pi_s$, it can be described with the same relations as above, but where all inequalities are made strict. 
Thus, $\Pi_s^{\circ}\cap\mathbb{Z}^{n+1}$ is the subset of $\Pi_s\cap\mathbb{Z}^{n+1}$ whose elements also satisfy
\begin{enumerate}
\item $0 < \frac{x_1}{s_1} ,$
\item $0 < \frac{x_{i+1}}{s_{i+1}} - \frac{x_i}{s_i}$ for $i \in[n-1]$, and
\item $0 < x_{n+1} - \frac{x_n}{s_n}.$
\end{enumerate}
To find REM$(\Pi_s^{\circ}\cap\mathbb{Z}^{n+1})$, we must see what restrictions the inequalities above translate to under the REM-function. 
The first inequality, $\frac{x_1}{s_1} > 0$, holds if and only if $x_1 \in [s_1-1]$, which is true if and only if REM$(x)_1 \neq 0$. 
For the second inequality, note that it can be written as
\[
0 < k_{i+1} + \frac{\text{REM}(x)_{i+1}}{s_{i+1}} - k_i - \frac{\text{REM}(x)_i}{s_i} < 1,
\]
for some non-negative integers $k_{i+1}$ and $k_i$. 
This holds if and only if 
\[
\frac{\text{REM}(x)_i}{s_i} \neq \frac{\text{REM}(x)_{i+1}}{s_{i+1},}.
\]
The third inequality is equivalent to $x_n \not | s_n$, which is equivalent to REM$(x)_n \neq 0$. 
Now we can conclude that REM$(\Pi^{\circ}_s\cap\mathbb{Z}^{n+1})$ is the subset of $I_n^s$ such that $e_1 \neq 0$, $\frac{e_i}{s_i} \neq \frac{e_{i+1}}{s_{i+1}}$ for $i \in [n-1]$, and $e_n \neq 0$.
Moreover, this subset is equal to $\widetilde{I^s_n}$. 
Finally, to prove that descents and ascents are equidistributed on $\widetilde{I^s_n}$ we note that the involution $f$ from the proof of Theorem \ref{thm: equidistributed} restricts to $\widetilde{I^s_n}$. 
This completes the proof.
\end{proof}

Throughout the remainder of this paper, for a sequence of positive integers $s = (s_1,\ldots,s_n)$ we will let $d_n^s(z):= \ell^\ast(P_n^s;z)$.  
In \cite{SV15}, the authors referred to the polynomials $A_n^s(z)$ as $s$-Eulerian polynomials since they generalize the classical Eulerian polynomial $A_n(z)$.  
Analogously, we will see in Section~\ref{sec: examples} that the polynomials $d_n^s(z)$ generalize the classical derangement polynomial $d_n(z)$. 
Thus, we will refer to $d_n^s(z)$ as the $s$-derangement polynomials.  
\begin{definition}
\label{def: s-derangement polynomial}
Let $s = (s_1,\ldots,s_n)$ be a sequence of positive integers.  The {\bf $s$-derangement polynomial} is
\[
d_n^s(z) := \ell^\ast(P_n^s;z).
\]
\end{definition}

\subsection{Real-rootedness}
\label{subsec: real-rootedness}
The Eulerian polynomial $A_n(z)$ is well-known to be symmetric, real-rooted, and thus log-concave and unimodal.  
In \cite{SV15}, it was shown that the $s$-Eulerian polynomials also possess these desirable distributional properties, with the only (possible) exception being symmetry.
In a similar fashion, the derangement polynomial $d_n(z)$ is also known to be symmetric, real-rooted, and thus log-concave and unimodal.
Since $s$-derangement polynomials are local $h^\ast$-polynomials of lattice simplices, then they are always symmetric.  
To show that they also possess these other nice distributional properties of $d_n(z)$, we show they are real-rooted via the theory of interlacing polynomials.  

Let $p(z) := p_0+p_1z+\cdots+p_dz^d\in\R[z]$ be a real-rooted polynomial of degree $d$ and suppose that it has roots $\alpha_1\geq\alpha_2\geq\cdots\geq\alpha_d$.  
Suppose also that $q(z) := q_0+q_1z+\cdots+q_mz^m\in\R[z]$ is a second real-rooted polynomial of degree $m$ with roots $\beta_1\geq\beta_2\geq\cdots\geq\beta_m$.  
We say that {\bf $q$ interlaces $p$}, denoted $q\preceq p$, if the roots of $p$ and $q$ are ordered such that
\[
\alpha_1\geq\beta_1\geq\alpha_2\geq\beta_2\geq\alpha_3\geq\beta_3\geq\cdots.
\]
Given a sequence of real-rooted polynomials $F :=\left(f^{(i)}\right)_{i=0}^n$, we say that $F$ is an {\bf interlacing sequence} of polynomials if $f_i\preceq f_j$ for all $0\leq i\leq j\leq n$.  
Notice that if a polynomial $p$ is a member of an interlacing sequence then it must be real-rooted.
Thus, to prove that a family of polynomials is real-rooted, it is typical to try and show that the polynomials of interest satisfy a recursion that produces a new interlacing sequence from an old one.
We will use the following such recursion.
\begin{lemma}
\label{lem: interlacing recursion}
Let $\left(f^{(i)}\right)_{i=0}^n$ be a sequence of interlacing polynomials.  
For $m\in\Z_{\geq0}$ and a map $\varphi:[m]_0\longrightarrow\Z_{\geq0}$ satisfying $\varphi(i)\leq\varphi(i+1)$ for all $i\in[m-1]_0$, define the polynomials
\[
g^{(i)} := z\sum_{j<\varphi(i)}f^{(j)}+\sum_{j\geq \varphi(i)}f^{(j)}.
\]
Let $\left(h^{(i)}\right)_{i=0}^m$ be a sequence of polynomials such that $h^{(i)}\in\{g^{(i)},g^{(i)}-f^{(\varphi(i))}\}$ for all $i\in[m]_0$.
If, for all $i\in[m]$, $h^{(i)}\neq g^{(i)}$ whenever $\varphi(i-1)=\varphi(i)$ and $h^{(i-1)} = g^{(i-1)}-f^{(\varphi(i))}$ then $\left(h^{(i)}\right)_{i=0}^m$ is also an interlacing sequence.
\end{lemma}

The proof of Lemma~\ref{lem: interlacing recursion} is immediate from \cite[Theorem 8.5]{B15}.  Lemma~\ref{lem: interlacing recursion} allows us to prove the following theorem.
\begin{theorem}
\label{thm: s-derangement real-rootedness}
Let $s= (s_1,\ldots, s_n)$ be a sequence of positive integers.  
Then the $s$-derangement polynomial $d_n^s(z)$ is real-rooted, and thus log-concave and unimodal.
\end{theorem}

\begin{proof}
Let ${\bf s} = (s_1,s_2,\ldots)$ be an infinite sequence of positive integers, and for each $n\geq1$ define the collection
\[
\widetilde{J_n^s} := 
\left\{
(e_0,\ldots,e_{n+1})\in\Z^{n+2}
\, : \,
0\leq e_i<\tilde{s}_i \mbox{ for $i\in[n+1]_0$}, \frac{e_i}{\tilde{s}_i}\neq\frac{e_{i+1}}{\tilde{s}_{i+1}} \mbox{ for $i\in[n-1]_0$}
\right\}, 
\]
where we assume $\tilde{s} = (s_0:=1, s_1,\ldots, s_n, s_{n+1}:=1)$ is defined with respect to the subsequence $s = (s_1,\ldots, s_n)$ of ${\bf s}$.  
For $n\geq 1$ and $0\leq k<s_n$, define the polynomials
\[
p_{n,k}^s(z) := \sum_{e\in\widetilde{J_n^s}}\chi(e_n = k)z^{\asc(e)},
\]
where $\chi(S)$ is equal to $1$ if $S$ is a true statement and $0$ otherwise.
Notice that $\widetilde{J_n^s}$ contains $\widetilde{I_n^s}$, and in particular, $e\in\widetilde{I_n^s}$ if and only if $e\in\widetilde{J_n^s}$ and $e_n\neq 0$.  
Thus,
\[
\ell^\ast(P_n^s;z) = \sum_{k=1}^{s_n-1}p_{n,k}^s(z).
\]
On the other hand, for $n\geq1$ and $0\leq k< s_n$, let $t_k:=\left\lceil\frac{ks_{n-1}}{s_n}\right\rceil$, and notice that
\begin{equation}
\label{eqn: recursion}
p_{n,k}^s(z) = z\sum_{i=0}^{t_k-1}p_{n-1,i}^s(z)+\chi(s_n\not\big| \, ks_{n-1})p_{n-1,t_k}^s(z) + \sum_{i=t_k+1}^{s_n-1}p_{n-1,i}^s(z).
\end{equation}
This recursion holds since if $e = (e_0,e_1,\ldots,e_n,e_{n+1})\in\widetilde{J_n^s}$ with $e_n = k$, then $n-1$ is a ascent in $e$ if and only if $\frac{e_{n-1}}{s_{n-1}}<\frac{k}{s_n}$.  
Equivalently, $n-1$ is an ascent in $e$ if and only if $e_{n-1}<t_k$.  
Moreover, for all $k\in[s_n-1]_0$,  it is impossible that $t_k = t_{k+1}$, $s_n\big|\, ks_{n-1}$, and $s_n\not\big|\, (k+1)s_{n-1}$.  
Therefore, the recursion in equation~\eqref{eqn: recursion} is of the form given in Lemma~\ref{lem: interlacing recursion}.
Moreover, for $n =1$ we have the initial conditions 
\[
p_{1,0}^s(z) = 0 
\qquad
\mbox{and} 
\qquad
p_{1,k}^s(z) = z
\mbox{ for $1\leq k<s_n$}.
\]
Thus, $(p_{1,k}^s)_{k=0}^{s_n-1}$ is an interlacing sequence.  
Lemma~\ref{lem: interlacing recursion} then implies that $(p_{n,k}^s)_{k=0}^{s_n-1}$ is an interlacing sequence for all $n\geq1$. 
Therefore, the subsequence $\ell_n^s := (p_{n,k+1}^s)_{k=0}^{s_n-2}$ is also an interlacing sequence, and applying the recursion in Lemma~\ref{lem: interlacing recursion} to $\ell_n^s$ with $m = s_n-2$ and $\varphi(i) := i$ for all $i\in[m]_0$ shows that 
\[
\sum_{k=0}^{s_n-2}p_{n,k+1}^s(z) = \ell^\ast(P_n^s;z)
\]
is a real-rooted polynomial.  
This completes the proof.
\end{proof}

\begin{remark}
\label{rmk: on the proof}
We remark that the recursion in Lemma~\ref{lem: interlacing recursion}, which we used in the proof of Theorem~\ref{thm: s-derangement real-rootedness}, has recently been used to prove a variety of real-rootedness results \cite{BL16,J16,SV15,S17}.
In particular, in the proof of Theorem~\ref{thm: s-derangement real-rootedness}, we essentially are showing that the recursion used by Savage and Visontai in \cite{SV15} to prove real-rootedness of $E_n^s(z)$ successfully restricts to the subsets $\widetilde{J_n^s}$ of $I_n^s$ and the associated polynomials $p_{n,k}^s(z)$.  
Furthermore, in a recent paper \cite{BL16}, Br\"and\'en and Leander generalized the $s$-Eulerian polynomials to $(P,s)$-Eulerian polynomials for a given poset $P$.  
Using Proposition~\ref{prop: formula for box polynomial}, it can be shown that the $s$-derangement polynomials are also $(P,s)$-Eulerian polynomials.  
An alternative proof of Theorem~\ref{thm: s-derangement real-rootedness} can then be derived from  \cite[Theorem 5.2]{BL16} which again uses the recursion in Lemma~\ref{lem: interlacing recursion}.  
Finally, we note that a third proof of Theorem~\ref{thm: s-derangement real-rootedness} that uses the theory of compatible polynomials \cite[Section 2.2]{SV15} can be found in \cite[Section 4.4]{G18}, the masters thesis of the first author (upon which this paper is based).  
At their heart, all of these proofs are fundamentally a consequence of Lemma~\ref{lem: interlacing recursion}, a commonly used recursion for recovering real-rootedness of generating polynomials.
\end{remark}

It turns out that the $s$-derangement polynomials provide an extensive and useful generalization of the derangement polynomial $d_n(z)$, similar to how the $s$-Eulerian polynomials provide such a generalization of the Eulerian polynomial $A_n(z)$.  
The remainder of this document is devoted to studying the nature and applications of this generalization.

\section{Examples}
\label{sec: examples}
In this section, we will observe that many of the well-studied families of derangement polynomials can be realized as $s$-derangement polynomials, much in the same fashion as how many of the well-studied generalizations of Eulerian polynomials can be realized as $s$-Eulerian polynomials.  
At the same time, these examples are also known to be local $h$-polynomials for well-studied subdivisions of a simplex in topology and algebra.  
Thus, we also provide a positive answer to Question~\ref{quest: local relations} by showing that these local $h$-polynomials are all local $h^\ast$-polynomials for $s$-lecture hall simplices.  

\subsection{Derangements}
\label{subsec: derangements}
Since $d_n(z)$ is the local $h$-polynomial of the barycentric subdivision of the $(n-1)$-simplex \cite[Proposition 2.4]{S92}, whose $h$-polynomial is the $n^{th}$ Eulerian polynomial, $A_n(z)$, we would analogously like to see that some $s$-lecture hall simplex with $h^\ast$-polynomial $A_n(z)$ has local $h^\ast$-polynomial $d_n(z)$.  
The first natural candidate for $P_n^s$ has $s = (1,2,\ldots,n)$.  
However, one can see from the inequalities defining the open parallelpiped $\Pi_s^\circ$ that $d_n^s(z) = 0$ since $s_1 = 0$.  
On the other hand, the reflexive $s$-lecture hall simplex with $s = (2,3,\ldots,n)$ also has $h^\ast$-polynomial $A_n(z)$, and this turns out to be exactly the $s$-lecture hall simplex we need.  
\begin{theorem}
\label{thm: derangements}
Let $s = (2,3,\cdots,n)$.  Then the $s$-derangement polynomial $d_n^s(z)$ is
\[
d_{n-1}^s(z) = d_n(z),
\]
the classical derangement polynomial.
Moreover, the local $h$-polynomial of the barycentric subdivision of the $(n-1)$-simplex is the local $h^\ast$-polynomial of the $s$-lecture hall simplex $P_{n-1}^s$.  
\end{theorem}

\begin{proof}
Note that by Proposition \ref{prop: formula for box polynomial}, it suffices to prove that
\[
\sum_{e \in \widetilde{I}^s_{n-1}}z^{\des(e)} = \sum_{\pi \in \mathfrak{D}_n}z^{\exc(\pi)}.
\]
This would follow if we could find a bijection $f : \widetilde{I}^s_{n-1} \rightarrow \mathfrak{D}_n$ such that $\des(e) = \exc(f(e))$. 
To do this, we will use the elements of $\widetilde{I}^s_{n-1}$ to build the cycles of permutations. 
Let $e \in \widetilde{I}^s_{n-1}$, and let $f(e)$ be the permutation obtained in the following way: 

First, set $A :=[n]$, the set of available elements, and keep track of the current cycle we are working on (at first, this cycle is empty). 
Go through the entries of $e$ in the order $e_n,e_{n-1},\ldots,e_1,e_0$. 
If $e_i = 0$ and the current cycle is nonempty, close off the current cycle, start a new cycle at $\min(A)$, the smallest element of $A$, and set $A:=A\setminus\{\min(A)\}$. 
If $e_i \neq 0$, add $\alpha$, the $e_i^{th}$ smallest number of $A$, to the current cycle and set $A:=A\setminus\{\alpha\}$. 
Note that the last element $e_0$ is always zero, so here we will always close off the last cycle and end the procedure. 
This process yields a permutation, and the map can be inverted. 
To invert it, take the smallest element from each cycle, sort the cycles according to these elements, and reverse the operations above. 
What remains to prove is that $f$ maps $\widetilde{I}^s_{n-1}$ to $\mathfrak{D}_n$, and that it maps descents to excedances. 

Recall that a descent in $e \in I^s_{n-1}$ is a pair of adjacent elements in $(e_0,e_1,\ldots,e_n)$ where $\frac{e_i}{s_i} > \frac{e_{i+1}}{s_{i+1}}$. 
There are two ways this can happen. 
First, either both $e_i$ and $e_{i+1}$ are non-zero and $e_i \geq e_{i+1}$. 
This corresponds to adding the $e_{i+1}^{st}$ smallest number to a cycle and the $e_i^{th}$ smallest number directly after. 
However, since $e_i \geq e_{i+1}$, the first number we add is smaller than the second. 
So this edge in the cycle decomposition in $f(e)$ will be an excedance. 
The second way a descent can happen is if $e_i$ is non-zero and $e_{i+1}$ is zero. 
This corresponds to starting a new cycle at the smallest number, and then connecting it to some bigger number. 
Again, this edge in the cycle decomposition is an excedance. 
Also, all excedances correspond to these descents, so $\exc(f(e)) = \des(e)$. 
Finally, the only way
$
\frac{e_i}{s_i} = \frac{e_{i+1}}{s_{i+1}}
$
can happen is if $e_{i} = e_{i+1} = 0$. So $\widetilde{I}^s_{n-1}$ is the set of $s$-inversion sequences where no two adjacent elements of $(e_0,e_1,\ldots,e_n)$ are zero. 
In $f(e)$, two adjacent zeros corresponds to starting a new cycle and then immediately closing it off. 
Therefore, $\widetilde{I}^s_{n-1}$ corresponds to the set of permutations in $\mathfrak{S}_n$ with no cycles of length $1$ (i.e.~fixed points), which is the set of derangements $\mathfrak{D}_n$.

Finally, the fact that the local $h$-polynomial of the barycentric subdivision of the $(n-1)$-simplex is $d_n^s(z)$ follows immediately from \cite[Proposition 2.4]{S92}.  
This completes the proof.
\end{proof}

\begin{example}
\label{ex: bijection example}
As an example of the bijection in the proof of Theorem \ref{thm: derangements}, suppose that $n = 5$ and $\widetilde{e} = (e_0,e_1,e_2,e_3,e_4,e_5) = (0,1,0,3,2,0)$. We will now read through $\widetilde{e}$ from right to left in order to get $f(e)$.

\begin{enumerate}
\item $e_5 = 0$, so we will start a cycle at the smallest element, which is $1$.
\item $e_4 = 2$, so we will add the second smallest element, $3$, to the cycle. 
\item $e_3 = 3$, so we will add the third smallest element, $5$, to the cycle.
\item $e_2 = 0$, so we will close off the current cycle and start a new one at the smallest element, $2$.
\item $e_1 = 1$, so we will add the smallest element, $4$, to the cycle.
\item $e_0 = 0$, so we will close off the current cycle, and since there are no more elements we are done.
\end{enumerate}
Therefore, $f(e)$ is given by the cycle decomposition $(1,3,5)(2,4)$, or equivalently, $f(e) = 34521$.
\end{example}

\begin{remark}
[On a geometric proof]
\label{rmk: not due to lattice geometry}
It follows from \cite[Remark 7.18]{KS16} that if $T$ is a lattice triangulation of an $s$-lecture hall simplex $P_n^s$ then the local $h^\ast$-polynomial $\ell^\ast(P_n^s;z)$ of $P_n^s$ and the local $h$-polynomial $\ell_{\verts(P_n^s)}(T_{\verts(P_n^s)};z)$ satisfy the monotonicity property
\[
\ell^\ast(P_n^s;z) \geq \ell_{\verts(P_n^s)}(T_{\verts(P_n^s)};z).
\]
Thus, a possible geometric proof of Theorem~\ref{thm: derangements} would be to show that for $s = (2,3,\ldots,n)$ the $s$-lecture hall simplex $P_{n-1}^s$ admits a lattice triangulation $T$ that is abstractly isomorphic to the barycentric subdivision of a simplex.  
One natural way to try and give such a geometric proof would be to first identify the vertices of $P_{n-1}^s$ with the vertices of the abstract simplicial complex $2^{[n]}$ prior to subdivision and then find a lattice triangulation of $P_{n-1}^s$ that induces a barycentric subdivision of $2^{[n]}$ given this identification.  
Unfortunately, as can be seen already in dimension two, $P_{n-1}^s$ always has edges that contain no interior lattice points (see Figure~\ref{fig: barycentric}), and hence this approach to a geometric proof of Theorem~\ref{thm: derangements} cannot work.  
However, there do exist $s$-lecture hall simplices  and local $h$-polynomials of subdivisions of a simplex for which this approach does work.  
We will see one such example in the coming subsection.
On the other hand, in Figure~\ref{fig: barycentric} we can also see that $P_2^{(2,3)}$ does admit a lattice triangulation that is abstractly isomorphic to the barycentric subdivision of $2^{[3]}$ if we do not identify the vertices of $P_2^{(2,3)}$ with those of $2^{[3]}$ prior to subdivision.  
Such a triangulation is given by coning over $P_2^{(2,3)}$ from the lattice point $(1,2)$.  
In general, and already for the case of $P_{n-1}^{(2,\ldots,n)}$ for $n\geq4$, it would be interesting to understand when a lattice $(n-1)$-simplex $\Delta$ having local $h^\ast$-polynomial $\ell^\ast(\Delta;z)$ equal to the local $h$-polynomial $\ell_{[n]}(\Omega_{[n]};z)$ for some subdivision $\Omega$ of $2^{[n]}$ has a lattice triangulation $T$ such that $\ell_{\verts(\Delta)}(T_{\verts(\Delta)};z) = \ell_{[n]}(\Omega_{[n]};z)$, either with or without the identification of $\Delta$ with $2^{[n]}$ prior to subdivision.
\end{remark}
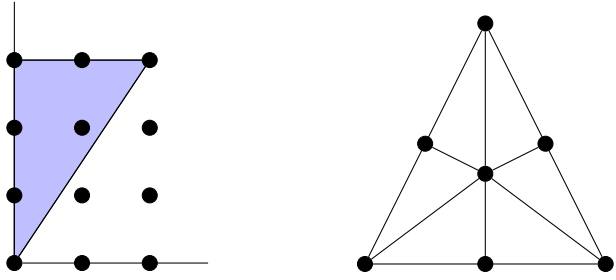
\begin{figure}
\label{fig: barycentric}
\centering
\begin{tabular}{c c}
\begin{tikzpicture}[scale=0.9]

\draw[fill = blue!25] (0,0) -- (0,3) -- (2,3) -- (0,0) -- cycle;
 	 \node [circle, draw, fill=black!100, inner sep=2pt, minimum width=2pt] (1) at (0,0) {};
 	 \node [circle, draw, fill=black!100, inner sep=2pt, minimum width=2pt] (2) at (0,1) {};
 	 \node [circle, draw, fill=black!100, inner sep=2pt, minimum width=2pt] (3) at (0,2) {};
 	 \node [circle, draw, fill=black!100, inner sep=2pt, minimum width=2pt] (4) at (0,3) {};
 	 \node [circle, draw, fill=black!100, inner sep=2pt, minimum width=2pt] (5) at (1,0) {};
 	 \node [circle, draw, fill=black!100, inner sep=2pt, minimum width=2pt] (6) at (1,1) {};
 	 \node [circle, draw, fill=black!100, inner sep=2pt, minimum width=2pt] (7) at (1,2) {};
 	 \node [circle, draw, fill=black!100, inner sep=2pt, minimum width=2pt] (8) at (1,3) {};
 	 \node [circle, draw, fill=black!100, inner sep=2pt, minimum width=2pt] (9) at (2,0) {};
 	 \node [circle, draw, fill=black!100, inner sep=2pt, minimum width=2pt] (10) at (2,1) {};
 	 \node [circle, draw, fill=black!100, inner sep=2pt, minimum width=2pt] (11) at (2,2) {};
 	 \node [circle, draw, fill=black!100, inner sep=2pt, minimum width=2pt] (12) at (2,3) {};

 	 \node (y) at (0,4) {};
 	 \node (x) at (3,0) {};

 	 \draw   	 (1) -- (4) ;
 	 \draw   	 (4) -- (12) ;
 	 \draw       	 (12) -- (1) ;
	
	 \draw   	 (1) -- (x) ;
 	 \draw   	 (1) -- (y) ;

\end{tikzpicture}

 	& \hspace{.5in}
	
\begin{tikzpicture}[scale=0.8]
 	 \node [circle, draw, fill=black!100, inner sep=2pt, minimum width=2pt] (1) at (0,-.5) {};
 	 \node [circle, draw, fill=black!100, inner sep=2pt, minimum width=2pt] (2) at (0,2) {};
 	 \node [circle, draw, fill=black!100, inner sep=2pt, minimum width=2pt] (3) at (-2,-2) {};
 	 \node [circle, draw, fill=black!100, inner sep=2pt, minimum width=2pt] (4) at (2,-2) {};
 	 \node [circle, draw, fill=black!100, inner sep=2pt, minimum width=2pt] (5) at (-1,0) {};
 	 \node [circle, draw, fill=black!100, inner sep=2pt, minimum width=2pt] (6) at (1,0) {};
 	 \node [circle, draw, fill=black!100, inner sep=2pt, minimum width=2pt] (7) at (0,-2) {};
	 
 	 \draw         (2) -- (3) ;
 	 \draw         (3) -- (4) ;
 	 \draw         (4) -- (2) ;
 
 	 \draw         (1) -- (2) ;
 	 \draw         (1) -- (3) ;
 	 \draw         (1) -- (4) ;
 	 \draw         (1) -- (5) ;
 	 \draw         (1) -- (6) ;
 	 \draw         (1) -- (7) ;
	 	  
\end{tikzpicture}
 	\\
\end{tabular}
\vspace{-0.2cm}
\caption{The $s$-lecture hall simplex $P_{n-1}^s$ with $s = (2,3,\ldots,n)$ does not admit a lattice triangulation that is isomorphic (as an abstract simplicial complex) to the barycentric subdivision of $2^{[n]}$ if we identify the vertices of $P_{n-1}^s$ with those of $2^{[n]}$ prior to subdivision.  This can be seen already for $n = 3$.}
\label{fig: barycentric subdivision} 
\end{figure}

\subsection{The $r^{th}$ Edgewise Subdivision}
\label{subsec: edgewise subdivision}
The $r^{th}$ edgewise subdivision is another well-studied subdivision of a simplicial complex that arises in a variety of contexts within algebra and topology \cite{BV09,BruR05,EG00,G89}.  
Within algebra, it is fundamentally connected to the Veronese construction \cite{BV09}. 
For $r\geq1$, the $r^{th}$ edgewise subdivision of a simplex is defined as follows:
Suppose that $\Delta\subset\R$ is an $(n-1)$-dimensional simplex with $0$-dimensional faces $e^{(1)},e^{(2)},\ldots,e^{(n)}\in\R^n$, the standard basis vectors in $\R^n$.  
For $x = (x_1,\ldots,x_n)\in\Z^n$, we let 
\[
\supp(x) := \{
i\in[n]
\, : \,
x_i\neq0
\},
\]
and we define 
\[
\varphi(x) := (x_1,x_1+x_2,\ldots,x_1+\cdots+x_n)\in\R^n.
\]
The {\bf $r^{th}$ edgewise subdivision} of $2^{\verts(\Delta)}$ is the simplicial complex $\left(2^{\verts(\Delta)}\right)^{\langle r\rangle}$ whose set of $0$-dimensional faces are the lattice points $r\Delta\cap\Z^n$, and for which $F\subset r\Delta\cap\Z^n$ is a face of $\left(2^{\verts(\Delta)}\right)^{\langle r\rangle}$ if and only if
\[
\bigcup_{x\in F}\{\supp(x)\}\in2^{[n]},
\]
and for all $x,y\in F$
\[
\varphi(x)-\varphi(y)\in\{0,1\}^n
\qquad
\mbox{or}
\qquad
\varphi(y)-\varphi(x)\in\{0,1\}^n.
\]
In the following, we will identify $2^{\verts(\Delta)}$ with $2^{[n]}$ prior to subdivision, and write $\left(2^{[n]}\right)^{\langle r\rangle}$ to denote the subdivision of $2^{[n]}$ induced by this identification and the $r^{th}$ edgewise subdivision of $2^{\verts(\Delta)}$.  
We analogously call $\left(2^{[n]}\right)^{\langle r\rangle}$ the {\bf $r^{th}$ edgewise subdivision }of $2^{[n]}$.

Let $\SW(n,r)$ denote the collection of all strings $\omega = \omega_0\omega_1\cdots\omega_n$ where $\omega_i\in[r-1]_0$, $\omega_0 = \omega_n = 0$, and $\omega_i \neq \omega_{i+1}$ for all $i = 0,\ldots, n$.  
The collection $\SW(n,r)$ is called the collection of all {\bf Smirnoff words} in \cite{A16} and \cite{LSW12}.  
Given a Smirnoff word $\omega\in\SW(n,r)$, we say that an index $i\in[n-1]_0$ is a {\bf descent} if $\omega_i>\omega_{i+1}$, and we let $\Des(\omega)$ denote the collection of all descents in $\omega$.  
Similarly, we say that an index $i\in[n-1]_0$ is an {\bf ascent} of $\omega$ if $\omega_i<\omega_{i+1}$, and we let $\Asc(\omega)$ denote the collection of all ascents in $\omega$. 
We also let $\des(\omega) := |\Des(\omega)|$ and $\asc(\omega):=|\Asc(\omega)|$.  

\begin{theorem}
\label{thm: edgewise subdivision}
For a positive integer $r$, let $s = (r,r,\ldots,r)$.  Then the $s$-derangement polynomial $d_n^s(z)$ is
\[
d_n^s(z) = \sum_{\omega\in\SW(n+1,r)}z^{\des(\omega)}.  
\]
Moreover, the local $h$-polynomial of the $r^{th}$ edgewise subdivision of the $n$-simplex is the local $h^\ast$-polynomial of the $s$-lecture hall simplex $P_{n}^s$.  
\end{theorem}

\begin{proof}
The fact that 
\[
d_n^s(z) = \sum_{\omega\in\SW(n+1,r)}z^{\des(\omega)}
\]
follows conveniently from the definition of $\SW(n+1,r)$ and that of $\widetilde{I_n^s}$.  
In \cite[Theorem 4.6]{A16}, it is shown that the local $h$-polynomial of the $r^{th}$ edgewise subdivision of $2^{[n+1]}$ is 
\[
\ell_{[n+1]}\left(\left(2^{[n+1]}\right)^{\langle r\rangle}_{[n+1]};z\right) = \sum_{\omega\in\SW(n+1,r)}z^{\asc(\omega)}.
\]
However, it is quick to see that ascents and descents in Smirnoff words are equidistributed via the involution
\[
\omega  = \omega_0\omega_1\ldots \omega_n\omega_{n+1}
\longmapsto 
\rev(\omega):=\omega_{n+1}\omega_n\ldots\omega_1\omega_0.
\]
This completes the proof.
\end{proof}

\begin{remark}
[A strictly geometric proof]
\label{rmk: geometric proof}
Unlike Theorem~\ref{thm: derangements}, Theorem~\ref{thm: edgewise subdivision} does admit a geometric proof as detailed in Remark~\ref{rmk: not due to lattice geometry}.  
In particular, note that, by the definition given, the geometric realization of the $r^{th}$ edgewise subdivision $\left(2^{[n+1]}\right)^{\langle r\rangle}$ of $2^{[n+1]}$ is a lattice triangulation of the convex $n$--simplex $r\Delta:=\conv\left(re^{(1)},\ldots,re^{(n+1)}\right)$. 
Projecting this triangulation along the last coordinate yields a lattice triangulation of the simplex
\[
\pi(r\Delta):=\conv\left(0,re^{(1)},\ldots,re^{(n)}\right)\subset\R^{n}.
\]
Via the unimodular transformation that maps $e^{(i)}\longmapsto e^{(1)}+\cdots+e^{(i)}$ for all $i\in[n]$, we see that $P_n^s$ is unimodularly equivalent to $\pi(r\Delta)$, and therefore admits a lattice triangulation that is abstractly isomorphic to $\left(2^{[n+1]}\right)^{\langle r\rangle}$.  
As described in Remark~\ref{rmk: not due to lattice geometry}, it then follows by \cite[Remark 7.18]{KS16} that $d_n^s(z)$ is the local $h$-polynomial of the $r^{th}$ edgewise subdivision of $2^{[n+1]}$ where we identified the vertices of $2^{[n+1]}$ with those of $P_n^s$ prior to subdivision.
An example of this geometric proof is depicted in Figure~\ref{fig: edgewise}.
\end{remark}
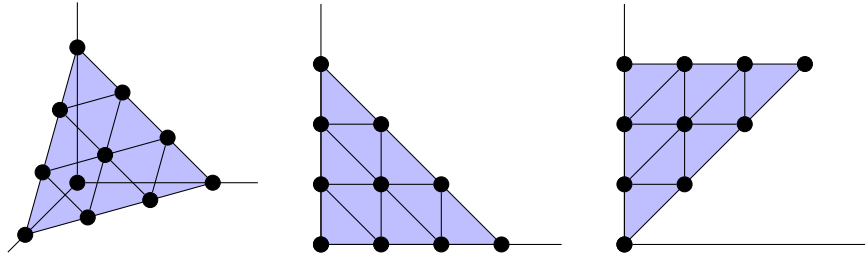
\begin{figure}
\label{fig: edgewise}
\centering
\begin{tabular}{c c c}
\begin{tikzpicture}[scale=0.6]

	\draw[fill=blue!25] (3,0,0) -- (0,3,0) -- (0,0,3) -- (3,0,0) -- cycle;

	
	\node [circle, draw, fill=black!100, inner sep=2pt, minimum width=2pt] (0) at (0,0,0) {};

	 \node [circle, draw, fill=black!100, inner sep=2pt, minimum width=2pt] (1) at (3,0,0) {};
	 \node [circle, draw, fill=black!100, inner sep=2pt, minimum width=2pt] (2) at (0,3,0) {};
	 \node [circle, draw, fill=black!100, inner sep=2pt, minimum width=2pt] (3) at (0,0,3) {};
	 \node [circle, draw, fill=black!100, inner sep=2pt, minimum width=2pt] (4) at (2,1,0) {};
	 \node [circle, draw, fill=black!100, inner sep=2pt, minimum width=2pt] (5) at (2,0,1) {};
	 \node [circle, draw, fill=black!100, inner sep=2pt, minimum width=2pt] (6) at (0,2,1) {};
	 \node [circle, draw, fill=black!100, inner sep=2pt, minimum width=2pt] (7) at (1,2,0) {};
	 \node [circle, draw, fill=black!100, inner sep=2pt, minimum width=2pt] (8) at (1,0,2) {};
	 \node [circle, draw, fill=black!100, inner sep=2pt, minimum width=2pt] (9) at (0,1,2) {};
	 \node [circle, draw, fill=black!100, inner sep=2pt, minimum width=2pt] (10) at (1,1,1) {};

 	 \draw   	 (0) -- (4,0,0) ;
 	 \draw   	 (0) -- (0,4,0) ;
 	 \draw   	 (0) -- (0,0,4) ;
	 \draw   	 (4) -- (9) ;
	 \draw   	 (6) -- (7) ;
	 \draw   	 (5) -- (6) ;
	 \draw   	 (7) -- (8) ;
	 \draw   	 (8) -- (9) ;
	 \draw   	 (4) -- (5) ;

\end{tikzpicture}

 	& \hspace{.1in}
	
\begin{tikzpicture}[scale=0.8]
	\draw[fill=blue!25] (3,0) -- (0,3) -- (0,0) -- (3,0) -- cycle;

	
	\node [circle, draw, fill=black!100, inner sep=2pt, minimum width=2pt] (0) at (0,0) {};

	 \node [circle, draw, fill=black!100, inner sep=2pt, minimum width=2pt] (1) at (3,0) {};
	 \node [circle, draw, fill=black!100, inner sep=2pt, minimum width=2pt] (2) at (0,3) {};
	 \node [circle, draw, fill=black!100, inner sep=2pt, minimum width=2pt] (3) at (0,0) {};
	 \node [circle, draw, fill=black!100, inner sep=2pt, minimum width=2pt] (4) at (2,1) {};
	 \node [circle, draw, fill=black!100, inner sep=2pt, minimum width=2pt] (5) at (2,0) {};
	 \node [circle, draw, fill=black!100, inner sep=2pt, minimum width=2pt] (6) at (0,2) {};
	 \node [circle, draw, fill=black!100, inner sep=2pt, minimum width=2pt] (7) at (1,2) {};
	 \node [circle, draw, fill=black!100, inner sep=2pt, minimum width=2pt] (8) at (1,0) {};
	 \node [circle, draw, fill=black!100, inner sep=2pt, minimum width=2pt] (9) at (0,1) {};
	 \node [circle, draw, fill=black!100, inner sep=2pt, minimum width=2pt] (10) at (1,1) {};

 	 \draw   	 (0) -- (4,0) ;
 	 \draw   	 (0) -- (0,4) ;
	 
	 \draw   	 (4) -- (9) ;
	 \draw   	 (6) -- (7) ;
	 \draw   	 (5) -- (6) ;
	 \draw   	 (7) -- (8) ;
	 \draw   	 (8) -- (9) ;
	 \draw   	 (4) -- (5) ;

\end{tikzpicture}

 	& \hspace{.1in}
	
\begin{tikzpicture}[scale=0.8]
	\draw[fill=blue!25] (3,3) -- (0,0) -- (0,3) -- (3,3) -- cycle;

	
	\node [circle, draw, fill=black!100, inner sep=2pt, minimum width=2pt] (0) at (0,0) {};

	 \node [circle, draw, fill=black!100, inner sep=2pt, minimum width=2pt] (1) at (3,3) {};
	 \node [circle, draw, fill=black!100, inner sep=2pt, minimum width=2pt] (2) at (0,0) {};
	 \node [circle, draw, fill=black!100, inner sep=2pt, minimum width=2pt] (3) at (0,3) {};
	 \node [circle, draw, fill=black!100, inner sep=2pt, minimum width=2pt] (4) at (2,2) {};
	 \node [circle, draw, fill=black!100, inner sep=2pt, minimum width=2pt] (5) at (2,3) {};
	 \node [circle, draw, fill=black!100, inner sep=2pt, minimum width=2pt] (6) at (0,1) {};
	 \node [circle, draw, fill=black!100, inner sep=2pt, minimum width=2pt] (7) at (1,1) {};
	 \node [circle, draw, fill=black!100, inner sep=2pt, minimum width=2pt] (8) at (1,3) {};
	 \node [circle, draw, fill=black!100, inner sep=2pt, minimum width=2pt] (9) at (0,2) {};
	 \node [circle, draw, fill=black!100, inner sep=2pt, minimum width=2pt] (10) at (1,2) {};

 	 \draw   	 (0) -- (4,0) ;
 	 \draw   	 (0) -- (0,4) ;
	 
	 \draw   	 (4) -- (9) ;
	 \draw   	 (6) -- (7) ;
	 \draw   	 (5) -- (6) ;
	 \draw   	 (7) -- (8) ;
	 \draw   	 (8) -- (9) ;
	 \draw   	 (4) -- (5) ;

\end{tikzpicture}
 	\\
\end{tabular}
\vspace{-0.2cm}
\caption{From left-to-right, the geometric realization of the $3^{rd}$ edgewise subdivision of the $2$-simplex, its unimodularly equivalent projection into $\R^2$, and the equivalent triangulation of the $s$-lecture hall simplex $P_n^s$ with $s = (r,r,\ldots,r)$ for $n = 2$ and $r=3$. }
\label{fig: barycentric subdivision} 
\end{figure}

\subsection{Derangements of colored permutations}
\label{subsec: derangements of colored permutations}
For integers $n,r\geq 1$, we say that a pair $(\pi,c)$ for $\pi\in\mathfrak{S}_n$ and $c\in\{0,\ldots,r-1\}^n$ is an {\bf $r$-colored permutation} (of length $n$).  
The value $c_i$ is called the {\bf color} of $\pi_i$ for each $i\in[n]$.  
We denote the collection of all $r$-colored permutations  (of length $n$) by $\Z_{r}\wr\mathfrak{S}_n$.  
The $r$-colored permutation $(\pi,c)$ is typically denoted as $\pi_1^{c_1}\pi_2^{c_2}\cdots\pi_n^{c_n}$.
We say that an index $i\in[n]$ is a {\bf descent} in $(\pi,c)$ if either $c_i>c_{i+1}$ or $c_i = c_{i+1}$ and $\pi_i>\pi_{i+1}$, where we assume that $\pi_{n+1} = n+1$ and $c_{n+1} = 0$.  
We also say that $i$ is an {\bf excedance} of $(\pi,c)$ if either $\pi_i>i$ or $\pi_i = i$ and $c_i>0$.  
We then denote the number of descents and excedances in $(\pi,c)$ by $\des(\pi,c)$ and $\exc(\pi,c)$, respectively.  
An $r$-colored permutation $(\pi,c)$ is called a {\bf derangement} if it has no fixed points of color $0$, and we denote the collection of all derangements in $\Z_{r}\wr\mathfrak{S}_n$ by $\mathfrak{D}_{n,r}$. 
Given these definitions, one can define $r$-colored analogues of the Eulerian and derangement polynomials.  
Specifically, define the polynomials
\[
A_{n,r}(z) :=\sum_{(\pi,c)\in\Z_{r}\wr\mathfrak{S}_n}z^{\des(\pi,c)}
\qquad
\mbox{and}
\qquad
d_{n,r}(z) :=\sum_{(\pi,c)\in\mathfrak{D}_{n,r}}z^{\exc(\pi,c)}. 
\]

Just like for regular permutations, descents and excedances are equidistributed in $\Z_{r}\wr\mathfrak{S}_n$.

\begin{theorem}
\cite[Theorem 3.15]{Ste92}
\label{theorem: colored permutations equidistributed}
For positive integers $n$ and $r$,
\[
A_{n,r}(z) = \sum_{\sigma \in \Z_{r}\wr\mathfrak{S}_n}z^{\text{exc}(\sigma)}.
\]
\end{theorem}
A useful corollary to this is that we can express the derangement polynomial in the following way.

\begin{corollary}
\label{cor: colored derangement formula}
For positive integers $n$ and $r$, 
$$d_{n,r}(z) = \sum_{k = 0}^n(-1)^{n-k}\binom{n}{k}A_{k,r}(z).$$
\end{corollary}

\begin{proof}
For some subset of indices $T \subseteq [n]$, let $\mathfrak{S}_{n,r,T}$ be the set of colored permutations such that $\pi_i = i$ and $c_i = 0$ for all $i \in T$. 
Using The Principle of Inclusion-Exclusion, we can write the derangement polynomial as
$$d_{n,r}(z) = \sum_{T \subseteq \{1,2,3,...,n\}}(-1)^{|T|}\sum_{\sigma \in \mathfrak{S}_{n,r,T}}z^{\text{exc}(\sigma)}.$$
Note that $\mathfrak{S}_{n,r,T}$ is in bijection with $\Z_r\wr\mathfrak{S}_{n-|T|}$ in a way that preserves the number of excedances; namely, by removing the elements whose indices are in $T$. Therefore,
$$d_{n,r}(z) = \sum_{T \subseteq \{1,2,3,...,n\}}(-1)^{|T|}\sum_{\sigma \in \Z_r\wr\mathfrak{S}_{n-|T|}}z^{\text{exc}(\sigma)}.$$
Combining this with Theorem \ref{theorem: colored permutations equidistributed} finishes the proof.
\end{proof}

In \cite{SV15}, it was shown that the colored Eulerian polynomials $A_{n,r}(z)$ are also realized as $s$-Eulerian polynomials:
\begin{proposition}
\cite[Lemma 3.5]{SV15}
\label{prop: colored eulerian polynomials}
For positive integers $n$ and $r$, $A_{n,r}(z)$ is equal to the $s$-Eulerian polynomial corresponding to $s = (r,2r,\ldots,nr)$.
\end{proposition}
Given Proposition~\ref{prop: colored eulerian polynomials}, we would then like to observe that the derangement polynomials $d_{n,r}(z)$ can similarly be computed using $s$-derangement polynomials.  
The $s$-derangement polynomial corresponding to $s = (r,2r,\ldots,nr)$ will not be equal to $d_{n,r}(z)$, but this is not surprising since $d_{n,r}(z)$ is not necessarily symmetric. 
Instead, we will express $d_{n,r}(z)$ as two $s$-derangement polynomials. 
To do this, we will also consider the sequence $s = (2r,\ldots,nr)$. 
This is in part motivated by subsection ~\ref{subsec: derangements}, where the sequence $s = (2,3,\ldots,n)$ gave the regular derangement polynomials $d_n(z)$.
\begin{theorem}
\label{theorem: colored derangement sum}
For positive integers $n$ and $r$, let $s = (2r,3r,\ldots,nr)$, and $\mu = (r,2r,\ldots,nr)$. 
Then
\[
d^s_{n-1}(z) + d^\mu_n(z) = d_{n,r}(z).
\]
\end{theorem}
To prove this, we will show that the sum $d^s_{n-1}(z) + d^\mu_n(z)$ can also be expressed as the right-hand-side of Corollary \ref{cor: colored derangement formula}. This can be done by interpreting $\widetilde{I^s_{n-1}}\bigcup\widetilde{I^\mu_n}$ as a subset of the colored permutations, and using The Principle of Inclusion-Exclusion. 
To formalize this argument, we will use the following definition.
\begin{definition}
For a colored permutation $\sigma = (\pi,c) \in \Z_{r}\wr\mathfrak{S}_n$, we will say that $i \in [n]$ is $\textbf{bad}$ with respect to $\sigma$ if for $\pi_j = i$ it holds that
\begin{enumerate}
\item $\pi_j < \pi_k$ for every $k > j$,
\item $\pi_{j-1} < \pi_k$ for every $k > j-1$, and
\item $\pi_j$ and $\pi_{j-1}$ have the same color.
\end{enumerate}
Here the convention $\pi_0 = 0$ and $c_0 = 0$ is used, and we let $S_{\sigma}$ be the set of bad numbers in $\sigma$.
\end{definition}
We then have the following lemma.
\begin{lemma}
\label{lemma: colored derangements with bad numbers}
For positive integers $n$ and $r$, let $s = (2r,3r,\ldots,nr)$, and $\mu = (r,2r,\ldots,nr)$. 
Then
\[
d^s_{n-1}(z) + d^\mu_n(z) = \sum_{\sigma \in \Z_{r}\wr\mathfrak{S}_n \, : \, S_\sigma = \emptyset}z^{\des(\sigma)}.
\]
\end{lemma}
\begin{proof}
The claim can be proved by interpreting elements in $\widetilde{I^s_{n-1}}$ and $\widetilde{I^\mu_n}$ as colored permutations using a map similar to $\Theta^{-1}$ from the proof of \cite[Lemma 3.5]{SV15}. 
Namely, for a permutation $\pi\in\mathfrak{S}_n$ and $i\in[n]$ we let 
\[
t_i :=|\{j > i \, : \, \pi_j<\pi_i\}|,
\]
denote the number of inversions of $\pi$ at $i$.  
The sequence $t = (t_1,\ldots,t_n)$ is called the inversion sequence (or Lehmer code) of $\pi$ and there exists a well-known bijection taking $t$ to $\pi$.  
We then define $\Psi: \Z_{r}\wr\mathfrak{S}_n\longrightarrow I_n^s$ where
\[
\Psi: \pi_1^{c_1}\cdots\pi_n^{c_n} \longmapsto (c_n+t_n,2c_{n-1}+t_{n-1},\ldots,nc_1+t_1).
\]
The inverse mapping $\Psi^{-1}$ is given by sending 
\[
\Psi^{-1}: (e_1,\ldots,e_n)\longmapsto \pi_1^{c_1}\cdots\pi_n^{c_n},
\]
where $\pi_1\cdots\pi_n$ is the permutation with inversion sequence
\[
t = (e_{n}-nc_1,e_{n-1}-(n-1)c_2,\ldots,e_1-c_n),
\]
and 
$
c_i = \left\lfloor\frac{e_{n-i+1}}{n-i+1}\right\rfloor
$ 
for each $i\in[n]$.
Next define $p:\widetilde{I^s_{n-1}}\rightarrow I^\mu_n$ as $p(e) = (0,e)$, and note that
\[
\Psi^{-1}(\widetilde{I^\mu_n}) = \{\sigma \in \Z_{r}\wr\mathfrak{S}_n \, : \,  S_\sigma = \emptyset \text{ and } c_n \neq 0 \},
\]
and
\[
\left(\Psi^{-1} \circ p\right)(\widetilde{I^s_{n-1}}) = \{\sigma \in \Z_{r}\wr\mathfrak{S}_n \, : \,  S_\sigma = \emptyset \text{ and } c_n = 0 \}.
\]
This follows from the definitions of $\widetilde{I^\mu_n}$ and $p(\widetilde{I^s_{n-1}})$. To prove it, pick an index $j \in \{1,2,\cdots,n\}$. Now it holds that $\frac{e_{n-j+1}}{(n-j+1)r} \neq \frac{e_{n-j+2}}{(n-j+2)r}$ for inversion sequences in both $\widetilde{I^\mu_n}$ and $p(\widetilde{I^s_{n-1}})$. This is true if and only if it does not hold that $c_j = c_{j-1}$ and $t_j = t_{j-1} = 0$ in $\Psi^{-1}(e)$, which is equivalent to $\pi_j$ not being bad. In $\widetilde{I^\mu_n}$, there is the additional constraint $0 \neq \frac{e_1}{r}$ which is equivalent to $c_n \neq 0$. On the other hand, in $p(\widetilde{I^s_{n-1}})$ we made sure that $c_n = 0$ from the definition of $p$.
This verifies the set equalities given above.
It then follows that 
\[
\Psi^{-1}(\widetilde{I^\mu_n}) \bigsqcup \Psi^{-1} \circ p(\widetilde{I^s_{n-1}}) = \{\sigma \in \Z_{r}\wr\mathfrak{S}_n \, : \, S_\sigma = \emptyset\}.
\]

Note that $\Psi$ has the property that $\asc(e) = \des(\Psi^{-1}(e))$, similar to \cite[Lemma 3.5]{SV15}.
It follows that $\asc(e) = \des(\Psi^{-1}(e))$ for $e \in \widetilde{I^\mu_n}$ and $\asc(e) = \asc(p(e)) = \des(\Psi^{-1}(p(e)))$ for $e \in \widetilde{I^s_{n-1}}$. Now we can conclude that
\begin{equation*}
d^s_{n-1}(z) + d^\mu_n(z) = \sum_{e \in \widetilde{I^s_{n-1}}}z^{\asc(e)} + \sum_{e \in \widetilde{I^\mu_n}}z^{\asc(e)}  = \sum_{\sigma \in \Z_{r}\wr\mathfrak{S}_n \, : \, S_\sigma = \emptyset}z^{\des(\sigma)}.
\end{equation*}
\end{proof}

With Lemma~\ref{lemma: colored derangements with bad numbers}, we are now ready to prove Theorem \ref{theorem: colored derangement sum}.

\begin{proof}[Proof of Theorem \ref{theorem: colored derangement sum}]
First, note that we can use Lemma \ref{lemma: colored derangements with bad numbers} to write
\[
	d^s_{n-1}(z) + d^\mu_n(z) = \sum_{\sigma \in \Z_{r}\wr\mathfrak{S}_n \, : \, S_\sigma = \emptyset}z^{\des(\sigma)} = \sum_{T \subseteq [n]}(-1)^{|T|}\sum_{\sigma \in \Z_{r}\wr\mathfrak{S}_n \, : \, T \subseteq S_\sigma}z^{\des{\sigma}}. 
\]
The second equality follows from The Principle of Inclusion-Exclusion. 
By Corollary~\ref{cor: colored derangement formula}, it now suffices to prove that
\begin{equation}
\label{eq: bad numbers eulerian}
\sum_{\sigma \in \Z_{r}\wr\mathfrak{S}_n \, : \, T \subseteq S_\sigma}z^{\des{\sigma}} = A_{n-|T|,r}(z).
\end{equation}
Let $\mathfrak{S}_{n,r,T}$ denote the set $\{\sigma \in \Z_{r}\wr\mathfrak{S}_n \, : \, T \subseteq S_\sigma\}$. To prove (\ref{eq: bad numbers eulerian}), we will find a bijection $f:\Z_r\wr\mathfrak{S}_{n-|T|} \rightarrow \mathfrak{S}_{n,r,T}$ such that $\des(\sigma) = \des(f(\sigma))$. 
Such a bijection $f$ is the following:

\begin{enumerate}
\item Take an element $\sigma = \pi_1^{c_1}\pi_2^{c_2}...\pi_{n-|T|}^{c_{n-|T|}} \in \Z_r\wr\mathfrak{S}_{n-|T|}$.
\item Replace each element $\pi_i = j$ with the $j^{th}$ smallest element of $[n]\setminus T$. (Note that this does not change the number of descents.)
\item We will now insert the numbers in $T$ into our half-finished permutation, in such a way that they become bad. Pick each $i\in T$ in order of size, starting with the smallest. If $i = 1$, insert $i$ at the front of $\sigma$ and give it color $0$. Otherwise, find the rightmost element $\pi_j$ such that $\pi_j < i$ and $\pi_j < \pi_k$ for every $k > j$. Give $i$ the same color as $\pi_j$ and insert it right after $\pi_j$.
\end{enumerate}

Notice first that $f$ is a well-defined function that maps to $\Z_{r}\wr\mathfrak{S}_n$. 
This follows if we can always find an index $j$ for every number $i \neq 1$, so that we can insert $i$ after $\pi_j$. Since we are inserting the numbers in order of size, when we insert $i \neq 1$, $1$ will already be in our permutation. Since $1$ satisfies the two conditions above, the set of indices $j$ where $\pi_j < i$ and $\pi_j < \pi_k$ for every $k > j$ is non-empty. Thus, the rightmost one exists.

Second, notice that $\des(f(\sigma)) = \des(\sigma)$. 
When we insert an element $i \neq 1$ after $\pi_j$, then no new descents are created since $\pi_j < i$. Also, if $(\pi_j,\pi_{j+1})$ was a descent earlier, then since $i$ gets the same color as $\pi_j$ and is greater, the pair $(i,\pi_{j+1})$ will still be a descent. If $i = 1$, then inserting it at the start and giving it color $0$ does not affect the number of descents.

Third, all the numbers in $T$ are bad in the resulting permutation. 
If $1 \in T$, then we will insert it at the start and give it color $0$, which makes it bad. 
After that, we will never insert anything in front of it, so it will always remain bad. For $i \neq 1$ in $T$, $i$ will be bad right after inserting it. 
The reason for this is that it will have the same color as the preceding element $\pi_j$, it will be greater than $\pi_j$, and $\pi_j$ will be smaller than all elements to the right of it. 
The only thing remaining to show is that $i$ is also smaller than everything to the right of it. 
Assuming that this is not true, pick the rightmost element $\pi_k$ such that $\pi_k < i$. 
This element must be smaller than everything to the right of it, but $\pi_j$ was the rightmost element satisfying this, which is a contradiction. 
Furthermore, the number $i$ will stay bad after the bijection is finished. 
This follows from the fact that we will never insert anything between $\pi_j$ and $i$ since all the elements inserted after $i$ will be greater than $i$. 
Also, inserting larger numbers into our permutation does not change the fact that $\pi_j$ or $i$ is smaller than everything to the right of it.

Finally, we can invert the function. To invert it, remove all numbers in $T$ from the permutation and undo step 2 from the description of the bijection. 
Thus, we conclude that $f:\Z_r\wr\mathfrak{S}_{n-|T|} \rightarrow \mathfrak{S}_{n,r,T}$ is a bijection for which $\des(f(\sigma)) = \des(\sigma)$.
By Corollary~\ref{cor: colored derangement formula}, this completes the proof. 
\end{proof}

\begin{example}
\label{ex: colored bijection example}
As an example of the bijection in the proof of Theorem \ref{theorem: colored derangement sum}, suppose that $n = 6$, $r = 3$, and $T = \{1,3,4\}$. 
Given an element $\sigma = 2^21^13^0 \in \Z_3\wr\mathfrak{S}_{3}$, applying step (2) in the description will transform it to $5^22^16^0$. 
Now we will insert the bad numbers $T = \{1,3,4\}$ one at a time as in step (3). 
First, the $1$ will end up at the front and get color zero, so the permutation becomes $1^05^22^16^0$. 
Next, the $3$ will be inserted after $2^1$ and get color $1$, so the permutation becomes $1^05^22^13^16^0$. 
Finally, the $4$ will be inserted after $3^1$ and get color $1$, so the final permutation is $1^05^22^13^14^16^0\in\Z_3\wr\mathfrak{S}_{6}$.
\end{example}

Theorem~\ref{theorem: colored derangement sum} demonstrates that the colored derangement polynomials $d_{n,r}(z)$ can also be realized in terms of $s$-derangement polynomials, similar to $d_n(z)$.  
However, since $d_{n,r}(z)$ is not symmetric in general, we needed to write it as a sum of two $s$-derangement polynomials, both of which are symmetric.  
In fact, this type of symmetric decomposition of a polynomial has recently been used both in the study of local $h$-polynomials \cite{A14} and in the study of unimodality of Ehrhart $h^\ast$-polynomials \cite{BS18,BJM16,S17}.  

Given a polynomial $p(z)\in\R[z]$ of degree at most $d$, it is an exercise in linear algebra to check that there exist unique polynomials $a(z),b(z)\in\R[z]$ such that 
\[
p(z) = a(z)+zb(z),
\]
where $a(z)$ is degree at most $d$, $b(z)$ is degree at most $d-1$, $a(z) = z^da(1/z)$, and $b(z) = z^{d-1}b(1/z)$.  
We call the ordered pair of polynomials $(a,b)$ the {\bf (symmetric) $\I_d$-decomposition} of $p$.  
The pair of $s$-derangement polynomials $(d_{n-1}^s,d_n^\mu)$ from Theorem~\ref{theorem: colored derangement sum} is indeed the $\I_n$-decomposition of $d_{n,r}(z)$.  
\begin{theorem}
\label{thm: colored derangements}
Let $r$ be a positive integer and let 
\[
s = (2r,3r,\cdots,nr)
\qquad
\mbox{and}
 \qquad
 \mu = (r,2r,3r,\ldots,nr).
 \]  
 The ordered pair of $s$-derangement polynomials $(d_{n-1}^s,d_n^\mu)$ is the $\I_n$-decomposition of $d_{n,r}(z)$.  
\end{theorem}

\begin{proof}
This result follows immediately from the uniqueness of $\I_d$-decompositions.  
In particular, if $\Delta$ is a $d$-dimensional simplex then its local $h^\ast$-polynomial satisfies
\[
\ell^\ast(\Delta;z) = z^{d+1}\ell^\ast\left(\Delta;\frac{1}{z}\right).
\]
In this case, since $d_{n-1}^s(z)$ and $d_n^\mu(z)$ are the local $h^\ast$-polynomials for the $s$-lecture all simplices $P_{n-1}^s$ and $P_n^\mu$, then they, respectively, satisfy
\[
d_{n-1}^s(z) = z^nd_{n-1}^s\left(\frac{1}{z}\right)
\qquad
\mbox{and}
\qquad
d_n^\mu(z) = z^{n+1}d_n^\mu\left(\frac{1}{z}\right).
\]
It follows by the uniqueness of $\I_d$-decompositions and Theorem~\ref{theorem: colored derangement sum} that $(d_{n-1}^s,d_n^\mu)$ is the $\I_n$-decomposition of $d_{n,r}(z)$.
\end{proof}
Our second main goal in this section has been to demonstrate that the local $h$-polynomials of well-studied subdivisions of a simplex can be realized as local $h^\ast$-polynomials of $s$-lecture hall simplices.  
In subsections~\ref{subsec: derangements} and~\ref{subsec: edgewise subdivision} we, respectively, saw that this is true for the barycentric and $r^{th}$ edgewise subdivision of a simplex.  
Combining these two, one consequence of Theorem~\ref{thm: colored derangements} is that the local $h$-polynomial of the $r^{th}$ edgewise subdivision of the barycentric subdivision of the $(n-1)$-simplex is the local $h^\ast$-polynomial of an $s$-lecture hall simplex.  
\begin{corollary}
\label{cor: colored local h-polynomial}
Let $r$ be a positive integer and let $s = (2r,3r,\cdots,nr).$
The local $h$-polynomial of the $r^{th}$ edgewise subdivision of the barycentric subdivision of $2^{[n]}$ is the local $h^\ast$-polynomial of the $s$-lecture hall simplex $P_{n-1}^s$.  
\end{corollary}

\begin{proof}
If $(a,b)$ denotes the $\I_n$-decomposition of $d_{n,r}(z)$, then by \cite[Theorem 1.2]{A14} with \cite[Theorem 1.3]{A14}, we can see that $a$ is the local $h$-polynomial of the $r^{th}$ edgewise subdivision of the barycentric subdivision of the $(n-1)$-simplex.  
Combining this observation with Theorem~\ref{thm: colored derangements} proves the result.
\end{proof}

\section{Applications}
\label{sec: applications}
In this section, we apply the results of Sections~\ref{sec: s-derangement polynomials} and~\ref{sec: examples} to address some open questions in the combinatorial literature.  
In subsection~\ref{subsec: symmetric decomposition} we affirmatively answer a conjecture posed in \cite[Question 4.11]{A16}, \cite[Conjecture 2.30]{A17}, and \cite[Conjecture 3.7.10]{S13}.
In subsection~\ref{subsec: order polytopes}, we show that all $s$-lecture hall order polytopes have a box unimodal triangulation, and use this result to provide a partial answer to a conjecture posed in \cite{BL16}.  
Similarly, in subsection~\ref{subsec: local order polytopes} we observe that all $s$-lecture hall order polytopes have unimodal local $h^\ast$-polynomials.

\subsection{Colored derangement polynomials}
\label{subsec: symmetric decomposition}
Our first application pertains to the symmetric decomposition of the colored derangement polynomial introduced in subsection~\ref{subsec: derangements of colored permutations}, and its relation to a well-studied conjecture on the $h$-vectors of flag simplicial complexes. 
A simplicial complex is called {\bf flag} if every minimal nonface contains two elements of the ground set. 
A topological subdivision $\Omega^\prime$ of a flag simplicial complex $\Omega$ is also called {\bf flag} if for every face $F\in\Omega$ the complex $\Omega^\prime_F$ is flag.  
A polynomial $p(z)\in\R[z]$ of degree at most $d$ is called {\bf $\gamma$-positive} (or {\bf $\gamma$-nonnegative}) \cite{A17} if when expressed in the basis $\left\{z^i(z+1)^{d-{2i}}\right\}_{i=0}^{\lfloor{d/2}\rfloor}$ for polynomials symmetric with respect to degree $d$ as
\[
p(z) = \sum_{i=0}^{\lfloor{d/2}\rfloor}\gamma_iz^i(z+1)^{d-2i}
\]
the coefficients $\gamma_0,\ldots,\gamma_{\lfloor{d/2}\rfloor}$ are all nonnegative.  
If $p(z)$ is the $h$-polynomial of a simplicial complex $\Omega$, then the vector of coefficients $\gamma(\Gamma) :=(\gamma_1,\ldots,\gamma_{\lfloor{d/2}\rfloor})$ is called the {\bf $\gamma$-vector} of $\Delta$.  
Gal's Conjecture \cite[Conjecture 2.1.7]{G05} claims that the $h$-polynomial of a flag homology sphere is always $\gamma$-nonnegative, and \cite[Conjecture 14.2]{PRW08}, \cite[Conjecture 1.4]{A12} further claimed that $\gamma(\Omega)\leq\gamma(\Omega^\prime)$ whenever $\Omega^\prime$ is a flag vertex-induced subdivision of a flag homology sphere $\Omega$.  
As seen from \cite[Proposition 5.3]{A12}, one can hope to use the $\gamma$-nonnegativity of local $h$-polynomials to prove the latter conjectures, and this idea was formalized in \cite[Conjecture 5.4]{A12}.  
One way to prove that a polynomial is $\gamma$-nonnegative is to show that it is both symmetric and real-rooted (see for instance \cite[Remark 3.1]{B15}).  
It is therefore of interest to know which local $h$-polynomials are real-rooted.  

Using combinatorial methods, \cite[Conjecture 5.4]{A12} has been verified for all subdivisions studied in Section~\ref{sec: examples} (see \cite[Section 4]{A16}).  
However, \cite[Question 4.11]{A16} further asked if the $\gamma$-nonnegativity of each of these local $h$-polynomials follows from real-rootedness.  
In \cite{Z95} real-rootedness of $d_n(z)$, the local $h$-polynomial of the barycentric subdivision of a simplex, is verified, and more recently in \cite{L16,Z16} the real-rootedness of the local $h$-polynomial of the $r^{th}$ edgewise subdivision of a simplex was shown.  
The results in Sections~\ref{sec: s-derangement polynomials} and~\ref{sec: examples} provide an alternative proof for these results, and also allow us to verify \cite[Question 4.11]{A16} for the final missing example.
Taking this one step further, \cite[Conjecture 2.30]{A17} and \cite[Conjecture 3.7.10]{S13}, claim that both polynomials in the $\I_n$-decomposition of $d_{n,r}(z)$ are real-rooted.  
The following theorem settles this collection of questions.  
\begin{theorem}
\label{thm: real-rooted symmetric decomposition}
If $(a,b)$ is the $\I_n$-decomposition of the colored derangement polynomial $d_{n,r}(z)$, then both $a$ and $b$ are real-rooted.  
In particular, the local $h$-polynomial of the $r^{th}$ edgewise subdivision of the barycentric subdivision of $2^{[n]}$ is real-rooted. 
\end{theorem}

\begin{proof}
The result follows by combining Theorem~\ref{thm: s-derangement real-rootedness} with Theorem~\ref{thm: colored derangements}.  
The latter statement is then seen by considering Corollary~\ref{cor: colored local h-polynomial}.
\end{proof}

\begin{remark}
[A second proof]
\label{rmk: with petter}
The real-rootedness of the $a$ and $b$ polynomials in the $\I_n$-decomposition of $d_{n,r}(z)$ was also recently observed, independently, by the second author and P.~Br\"and\'en \cite{BS18}.  
There, the authors actually prove the stronger statement that the $a$ and $b$ polynomials alway interlace.  
On the other hand, the proof given in Theorem~\ref{thm: real-rooted symmetric decomposition} has the advantage that the same proof of \cite[Question 4.11]{A16} works for all local $h$-polynomials discussed above.
\end{remark}

We end this subsection with one more note: 
Here we introduced the additional distributional property of $\gamma$-nonnegativity.  
It is worthwhile to note that since any symmetric and real-rooted polynomial is $\gamma$-nonnegative, then all $s$-derangement polynomials are $\gamma$-nonnegative by Theorem~\ref{thm: s-derangement real-rootedness}.
\begin{corollary}
\label{cor: gamma-nonnegativity}
All $s$-derangement polynomials are $\gamma$-nonnegative.
\end{corollary}

\subsection{Ehrhart $h^\ast$-polynomials of $s$-lecture hall order polytopes}
\label{subsec: order polytopes}
In a recent paper \cite{BL16}, Br\"and\'en and Leander introduced a family of lattice polytopes that simultaneously generalize two well-studied families of polytopes: the $s$-lecture hall simplices and the order polytopes.
Analogously, they refer to these polytopes as $s$-lecture hall order polytopes.  
In this subsection, we will use the results of Section~\ref{sec: s-derangement polynomials} to prove that all $s$-lecture hall order polytopes have a box unimodal triangulation.  
Via Theorem~\ref{thm: reflexive}, we then find that all reflexive $s$-lecture hall order polytopes have unimodal $h^\ast$-polynomials.  
As a corollary to this result, we are able to provide a partial answer to a conjecture posed in \cite{BL16}.  

In the following, we will let $P = ([n];\preceq_P)$ denote a labeled poset on ground set $[n]$, and we use $\leq$ to denote the typical total order placed on $\Z$.  
We say that $P$ is {\bf naturally labeled} if $i\preceq_P j$ whenever $i\leq j$.  
For a naturally labeled poset $P$ the {\bf order polytope} of $P$ is 
\[
O(P):=
\left\{
(x_1,\ldots,x_n)\in\R^n 
\, : \,
0\leq x_i\leq 1, i\in[n], \mbox{ and } x_i\leq x_j \mbox{ if } i\preceq_P j 
\right\}.
\]
Given a sequence of positive integers $s = (s_1,\ldots,s_n)$, the {\bf $s$-lecture hall order polytope} of the pair $(P,s)$ is 
\[
O(P,s):=
\left\{
(x_1,\ldots,x_n)\in\R^n 
\, : \,
0\leq x_i\leq s_i, i\in[n], \mbox{ and } \frac{x_i}{s_i}\leq \frac{x_j}{s_j} \mbox{ if } i\preceq_P j 
\right\}.
\]
A lattice triangulation $T$ of a lattice polytope is called {\bf unimodular} if the normalized volume of each simplex $\Delta\in T$ (i.e.~the value $h^\ast(\Delta;1)$) is equal to $1$.  
It is well-known that for any naturally labeled poset $P$, the order polytope $O(P)$ admits a regular and unimodular triangulation $T(P)$ known in the literature as the {\bf canonical triangulation} of $O(P)$ (see for instance \cite{RW05}).  
In \cite[Lemma 3.1]{BL16} it is shown that any $s$-lecture hall order polytope $O(P,s)$ admits a triangulation $T(P,s)$ whose facets are all $s$-lecture hall simplices.  
Moreover, it can be seen that $T(P,s)$ is simply a scaling of $T(P)$ by $s_i$ in each coordinate $i\in[n]$, and therefore it is also a regular triangulation.  
In the following, we will call the triangulation $T(P,s)$ the {\bf $s$-canonical triangulation} of $O(P,s)$. 

Our first goal in this subsection is to prove that the $s$-canonical triangulation of an $s$-lecture hall order polytope is alway box unimodal.  
To do this, we will need the following lemma.
\begin{lemma}
\label{lem: gcd}
Let $s = (s_1,\ldots,s_n)$ be a sequence of positive integers, set $g:=\gcd(s_1,\ldots,s_n)$, and let $x = (x_1,\ldots,x_n)$ be a sequence of nonnegative integers satisfying
\[
\frac{x_1}{s_1} = \frac{x_2}{s_2} = \cdots = \frac{x_n}{s_n}.
\]
Then $\frac{gx_i}{s_i}$ is an integer for all $i\in[n]$. 
\end{lemma}

\begin{proof}
By symmetry it suffices to prove that $\frac{gx_1}{s_1}$ is an integer. 
We do so via an inductive argument. 
Assume first that $n = 2$. 
Now 
\[
x_2 = \frac{s_2x_1}{s_1} = \frac{\frac{s_2}{g}x_1}{\frac{s_1}{g}}
\]
is an integer.
Since $\frac{s_2}{g}$ and $\frac{s_1}{g}$ are co-prime then $x_1$ must be divisible by $\frac{s_1}{g}$, which implies that 
\[
\frac{x_1}{\frac{s_1}{g}} = \frac{gx_1}{s_1}
\]
is an integer. 
Now take $n > 2$, let $g' = \text{gcd}(s_1,...,s_{n-1})$, and assume that the statement is true for $s_1,...,s_{n-1}$. 
This implies that $\frac{g'x_1}{s_1}$ is an integer. Thus, 
\[
\frac{\frac{g'x_1}{s_1}}{g'} = \frac{x_n}{s_n},
\]
and we can use the result for $n = 2$ on this new sequence to conclude that 
$$\frac{\text{gcd}(g',s_n)\frac{g'x_1}{s_1}}{g'} = \frac{gx_1}{s_1}$$
is an integer.
\end{proof}

Since the $s$-canonical triangulation of $O(P,s)$ is always regular, then to prove it is a box unimodal triangulation we need only show that each simplex in $T(P,s)$ has a unimodal local $h^\ast$-polynomial.  
To see this, we prove the following theorem.
\begin{theorem}
\label{thm: faces}
Let $s = (s_1,\ldots, s_n)$ be a sequence of positive integers and let $F$ be a face of the $s$-lecture hall simplex $P_n^s$.
Then $\ell^\ast(F;z) = d_n^{\mu}(z)$ for some sequence of positive integers $\mu$.
\end{theorem}

\begin{proof}
If $F$ is a vertex, then $\ell^\ast(F;z) = 0$, so the theorem holds. 
Now assume that $F$ is not a vertex. 
Note that $P^s_n$ can be written as

\[
P^s_n = \text{conv} \left (  \begin{pmatrix}s_1 \\ s_2 \\ \vdots \\ s_n\end{pmatrix} ,\begin{pmatrix}0 \\ s_2 \\ \vdots \\ s_n\end{pmatrix}, \cdots, \begin{pmatrix}0 \\ \vdots \\ 0 \\ s_n\end{pmatrix} , \begin{pmatrix}0 \\ 0\\ \vdots \\ 0\end{pmatrix}    \right),
\]
and set $v_i :=(0,0,\ldots,0,s_{i+1},\ldots,s_n)^T$ for $i\in[n]_0$.
Let
\[
F = \text{conv}(v_{i_0},v_{i_1},\cdots,v_{i_m}),
\]
where $0 \leq i_0 < i_1 < \cdots < i_m \leq n$ be a face of $P_n^s$ with $m>0$.
Let $\mu = (\mu_1,\mu_2,\cdots,\mu_m)$ be the sequence of positive integers defined by
\[
\mu_j := \text{gcd}(s_{i_{j-1}+1}, s_{i_{j-1}+2},\cdots,s_{i_j}),
\]
for $j \in[m]$. 
To prove the theorem is suffices to find a bijection 
\[
f:\Pi^\circ_F \cap \mathbb{Z}^{n+1} \rightarrow \Pi^\circ_{P^\mu_m}\cap\mathbb{Z}^{m+1}
\]
such that $x_{n+1} = f(x)_{m+1}$. 
To do so, take a point $x \in \Pi^{\circ}_{F} \cap \mathbb{Z}^{n+1}$, and note that it can be written as 
\[
x = \begin{pmatrix}x_1 \\ x_2 \\ ... \\ x_{n+1}\end{pmatrix} = \sum_{j = 0}^m\lambda_{j}w_{i_j} = \begin{pmatrix}s_1q_1 \\ s_2q_2 \\ ...\\ s_nq_n \\ \sum_{j = 0}^m\lambda_j\end{pmatrix},
\]
where $w_k := (v_k,1)$, $\lambda_k \in (0,1)$, and 
\[
q_k := \sum_{i_j < k}\lambda_j.
\]
Let $f(x)$ be the point in $\Pi_{P^\mu_m}$ defined by
$$f(x) :=  \sum_{j = 0}^m\lambda_jw^*_j,$$
where $w^*_j = (v_j^*,1)$ and $v_j^*$ is the $j^{th}$ vertex of $P^\mu_m$. This is an injective function, and since $f(x)$ can be written as
$$f(x) = \begin{pmatrix}\lambda_0\mu_1 \\ (\lambda_0 + \lambda_1)\mu_2 \\ ... \\ (\lambda_0 + ... + \lambda_{m-1})\mu_m \\ \sum_{j = 0}^m\lambda_j \end{pmatrix},$$
its last coordinate is the same as the last coordinate of $x$. What remains to show is that $f$ is well-defined (maps to $\Pi_{P^\mu_m}^\circ \cap \mathbb{Z}^{m+1})$ and surjective. For $j \in[m]$, note that
\[
q_{i_{j-1}+1} = q_{i_{j-1}+2} = ... = q_{i_j},
\]
and set $q_j:= q_{i_j}$.  
Since $x \in \mathbb{Z}^{n+1}$, 
\[
q_js_{i_{j-1}+k} = x_{i_{j-1}+k} \in \mathbb{Z},
\]
for $k \in [i_j-i_{j-1}]$. 
Thus, 
\[
q_j = \frac{x_{i_{j-1}+1}}{s_{i_{j-1}+1}} = \frac{x_{i_{j-1}+2}}{s_{i_{j-1}+2}} = ... = \frac{x_{i_j}}{s_{i_j}}.
\]
Using Lemma \ref{lem: gcd} on this we see that $\mu_jq_j$ is an integer. 
However, note that $f(x)_j = \mu_jq_j$ for $j \in [m]$. 
Thus, $f(x)$ is in $\mathbb{Z}^{m+1}$. 
Also, since $x \in \Pi_{F}^{\circ}$, the numbers $\lambda_0, ..., \lambda_m$ are contained in $(0,1)$. 
Therefore, $f(x)$ is in the interior of $\Pi_{P^\mu_m}^\circ$, which means that it maps to $\Pi_{P^\mu_m}^{\circ} \cap \mathbb{Z}^{m+1}$. Furthermore, for a particular integer $f(x)_j = \mu_jq_j$, we can let $x_{i_{j-1}+k}$ equal
\[
x_{i_{j-1}+k} = \mu_jq_j\frac{s_{i_{j-1}+k}}{\mu_j},
\]
for $k \in [i_j-i_{j-1}]$, in order to map to $f(x)_j$. 
The numbers $x_{i_{j-1}+k}$ will then be integers thanks to the fact that $\mu_j$ is a divisor of $s_{i_{j-1}+k}$. 
This proves that $f$ is surjective, thereby completing the proof.
\end{proof}

The real-rootedness of the local $h^\ast$-polynomial of each simplex in the $s$-canonical triangulation of $O(P,s)$ follows immediately by combining Theorem~\ref{thm: faces} and Theorem~\ref{thm: s-derangement real-rootedness}.  
\begin{corollary}
\label{cor: faces real-rootedness}
Let $s = (s_1,\ldots, s_n)$ be a sequence of positive integers and let $F$ be a face of the $s$-lecture hall simplex $P_n^s$.
Then the local $h^\ast$-polynomial $\ell^\ast(F;z)$ of $F$ is real-rooted, and thus unimodal.
\end{corollary}

\begin{proof}
The result follows immediately by combining Theorem~\ref{thm: s-derangement real-rootedness} with Theorem~\ref{thm: faces}.
\end{proof}

As an immediate consequence of Corollary~\ref{cor: faces real-rootedness} and our results in Section~\ref{sec: s-derangement polynomials}, we recover our first two desired conclusions of this subsection.  
\begin{theorem}
\label{thm: box unimodal}
Let $O(P,s)$ be an $s$-lecture hall order polytope.  
Then the $s$-canonical triangulation of $O(P,s)$ is box unimodal.  
Moreover, if $O(P,s)$ is reflexive then its $h^\ast$-polynomial $h^\ast(O(P,s);z)$ is unimodal.
\end{theorem}

\begin{proof}
Suppose first that $O(P,s)$ is any $s$-lecture hall order polytope.  
By Corollary~\ref{cor: faces real-rootedness}, we know that for each face $F\in T(P,s)$, the $s$-canonical triangulation of $O(P,s)$, has a real-rooted, and therefore unimodal, local $h^\ast$-polynomial.  
Since the $s$-canonical triangulation of $O(P,s)$ is always regular, it follows that it is box unimodal.  
In the special case that $O(P,s)$ is reflexive, it then follows from Theorem~\ref{thm: reflexive}, and the regularity of $T(P,s)$, that $O(P,s)$ has a unimodal $h^\ast$-polynomial.
\end{proof}

Recently, there has been a growing interest in the identification of lattice polytopes with box unimodal triangulations \cite{B16,Ober17,SV13}, specifically because of its applications to unimodality questions for $h^\ast$-polynomials of reflexive polytopes.
Theorem~\ref{thm:  box unimodal} demonstrates that a large family of lattice polytopes containing other well-studied polytopes admit this desirable property.  

As an application of Theorem~\ref{thm: box unimodal}, we can additionally provide a partial answer to a conjecture posed in \cite{BL16}.  
To do so, we first recall some fundamental definitions for posets.  
A {\bf chain} is a poset in which any two elements are regarded as comparable.  
If $x_1\prec x_2\prec \cdots \prec x_n$ is a chain then the {\bf rank} of $x_i$ in the chain is $i-1$ for all $i\in[n]$.  
A subposet $C$ of a poset $P$ is called a {\bf chain} in $P$ if $C$ is a chain when it is regarded as a subposet of $P$.  
The chain $C$ is called {\bf maximal} if it is not contained in a larger chain within $P$.  
The {\bf length} of a chain $C$ is the number of elements in $C$ minus one.
A poset $P$ is called {\bf graded} if every maximal chain in $P$ has the same length.
A poset $P$ is called {\bf ranked} if for every maximal element $x$ in $P$ the subposet $\{y\in P: y\preceq_P x\}$ is graded. 
If $P$ is a {\bf ranked} poset, then the {\bf rank} of any $x\in P$ is its rank in any maximal chain of $P$.
In this case we let $\rho_P: [n]\longrightarrow\Z_{\geq0}$ denote the function mapping each $x\in P$ to its rank.  
In \cite{BL16} the authors proved the following.
\begin{theorem}
\label{thm: branden and leander}
\cite[Theorem 4.2]{BL16}
Let $P = ([n],\preceq_p)$ be a naturally-labeled ranked poset and let $s = (s_1,\ldots, s_n)$ be the sequence of positive integers where 
\[
s_i:=\rho_P(i)+1.
\]
Then $h^\ast(O(P,s);z)$ is symmetric with respect to its degree, which is $n-1$.  
\end{theorem}

In \cite[Section 5]{BL16} they further ask if the $h^\ast$-polynomials in Theorem~\ref{thm: branden and leander} are also unimodal.  
Specifically, this is the motivating question behind the statement of \cite[Conjecture 5.4]{BL16}. 
Using Theorem~\ref{thm: box unimodal} we can provide the following partial answer to this question, specifically in the case when the poset $P$ has a unique minimal element.
\begin{corollary}
\label{cor: posets with a unique minimal element}
Let $P= ([n],\preceq_P)$ be a naturally-labeled ranked poset with rank function $\rho_P$ and let $s = \rho_P+1$.
If $P$ has a unique minimal element then $h^\ast(O(P,s);z)$ is unimodal.
\end{corollary}

\begin{proof}
Let $P= (\{0\}\cup[n],\preceq_P)$ be a naturally-labeled ranked poset with a unique minimal element $0$ and rank function $\rho_P:[n]_0\longrightarrow\Z_{\geq0}$.  
Since $O(P,s)$ is always $n$-dimensional for a poset on $n$ elements, then by Theorem~\ref{thm: box unimodal} it suffices to prove that 
\[
h^\ast(O(P,s);z) = h^\ast(O(Q,s^\prime);z),
\]
where $Q$ denotes the subposet of $P$ on the elements $[n]$ and $s^\prime = (\rho_P+1)\big|_{[n]}$.  
However, this fact follows immediately from \cite[Corollary 3.7]{BL16}, and so the result holds.
\end{proof}

\begin{remark}
It is important to note that the conjectured unimodality of all the $h^\ast$-polynomials considered in Theorem~\ref{thm: branden and leander} cannot be recovered via Theorem~\ref{thm: box unimodal} alone.  
This is because some $h^\ast$-polynomials from Theorem~\ref{thm: branden and leander} cannot be the $h^\ast$-polynomial of a reflexive $O(P,s)$.  
For example, consider the poset $P = ([3],\preceq_P)$ in which $1\prec_P3$ and $2\prec_P3$.  
Then $P$ is naturally labeled with rank function $\rho$ and
\[
h^\ast(O(P,s);z) = 1+2z+z^2,
\]
where $s=\rho+1$.  
If we want to realize $1+2z+z^2$ as the $h^\ast$-polynomial of some reflexive $2$-dimensional $O(Q,\mu)$ then we know that $Q$ must be a naturally-labeled poset with ground set $[2]$ and that the normalized volume of $O(Q,\mu)$ must equal $4$.  
The only three possibilities for such an $O(Q,\mu)$ are
\begin{enumerate}
	\item $Q$ has no relations with $\mu_1 = 1$ and $\mu_2 = 2$ (or $\mu_1 = 2$ and $\mu_2 = 1$),
	\item $Q$ has the single relation $1\prec_Q2$ and $\mu_1 = 1$ and $\mu_2 = 4$ (or $\mu_1 = 4$ and $\mu_2 = 1$), or 
	\item $Q$ has the single relation $1\prec_Q2$ and $\mu_1 = \mu_2 = 2$.  
\end{enumerate}
However, each of these $O(Q,\mu)$ have $h^\ast$-polynomial $1+3z$. 
Thus, further work beyond Theorem~\ref{thm: box unimodal} is needed to fully answer the unimodality question posed in \cite{BL16} in relation to \cite[Conjecture 5.4]{BL16}. 
\end{remark}

\subsection{Local $h^\ast$-polynomials of $s$-lecture hall order polytopes}
\label{subsec: local order polytopes}
In \cite[Example 7.13]{S92}, Stanley generalized the definition of a local $h^\ast$-polynomial of a lattice simplex given by Betke and McMullen \cite{BM85} to arbitrary lattice polytopes.  
Given a $d$-dimensional lattice polytope $P$, we define the face poset of $P$, denoted $\mathcal{F}(P)$, to be the poset whose ground set is the collection of all faces of $P$ (including the empty set) with the partial order $\preceq_{\mathcal{F}(P)}$ being inclusion.  
In the following, we denote the set of all faces of $P$ by $F(P)$.  
Notice that $\mathcal{F}(P)$ is a ranked poset with rank function $\rho(F) = \dim(F)+1$ for all faces $F$ of $P$.  
The {\bf $g$-polynomial} of $P$ \cite{S11} is defined recursively in the following way: 
If $P$ is the empty polytope, i.e.,~$P = \emptyset$, then $g(P;z) = 1$.  
If $P$ is dimension $d$, then $g(P;z)$ is the unique polynomial of degree strictly less than $\frac{d}{2}$ satisfying
\[
z^dg\left(P\, ;\frac{1}{z}\right) = \sum_{F\in F(P)}g(F;z)(t-1)^{d-\dim(F)-1}.
\]
Given a poset $P = ([n];\preceq_P)$, we let $P^\vee := ([n];\preceq_{P^\vee})$ denote its {\bf dual} poset; i.e.,~the poset for which $i\preceq_{P^\vee}j$ whenever $j\preceq_P i$.  
The {\bf local $h^\ast$-polynomial} of $P$ is then 
\begin{equation}
\label{eqn: general local}
\ell^\ast(P;z) = \sum_{F\in F(P)}(-1)^{\dim(P)-\dim(F)}h^\ast(F;z)g([F,P]^\vee;z).
\end{equation}
Notice that if $P$ is a simplex then $g(P;z) = 1$ (see for instance \cite[Example 3.9]{KS16} or \cite[Remark 4.2]{BN08}).  
Thus, for a simplex $P$, the formula for $\ell^\ast(P;z)$ given in equation~\eqref{eqn: general local} reduces to that in equation~\eqref{eqn: local} by applying The Principle of Inclusion-Exclusion.  

In subsection~\ref{subsec: order polytopes} we saw an example of how the symmetry and unimodality of local $h^\ast$-polynomials in a regular lattice triangulation of a lattice polytope $P$ can be used to recover unimodality of the $h^\ast$-polynomial of $P$.  
Unfortunately, in practice, we are typically restricted to assuming other sufficient conditions, such as reflexivity of $P$, in order to prove that $h^\ast(P;z)$ is unimodal by way of a box unimodal triangulation.  
On the other hand, it follows from \cite[Remark 7.23]{KS16} that these additional assumptions are not required if we only wish to prove unimodality of the local $h^\ast$-polynomial $P$.  
In particular, we have the following result.
\begin{theorem}
\label{thm: unimodal local order polytopes}
The local $h^\ast$-polynomial of an $s$-lecture hall order polytope is unimodal.  
\end{theorem}

\begin{proof}
The result follows immediately by combining Theorem~\ref{thm: box unimodal} with \cite[Lemma 7.12(4)]{KS16} and \cite[Theorem 6.1]{KS16}.  
\end{proof}

\begin{remark}
\label{rmk: more generally speaking}
More generally speaking, the same proof shows that if $P$ is a lattice polytope with a box unimodal triangulation then its local $h^\ast$-polynomial is unimodal.  
This can be viewed as somewhat of a ``local'' analogue to the well-known theorem of Bruns and R\"omer which states that if $P$ is a Gorenstein polytope with a regular and unimodular triangulation then its $h^\ast$-polynomial is unimodal \cite[Theorem 1]{BrR07}.
\end{remark}

\section{Final Remarks}
\label{sec: final remarks}

In this note, we observed that the local $h^\ast$-polynomials, or box polynomials, of a family of lattice simplices known as the $s$-lecture hall simplices generalize well-studied families of derangement polynomials in the combinatorial literature.  
Moreover, this generalization preserves all of the desirable distributional properties of the classical derangement polynomial; namely, symmetry, real-rootedness, log-concavity, unimodality, and $\gamma$-nonnegativity.  
Using these results, we showed that the local $h$-polynomials of a number of well-studied subdivisions of a simplex can be realized as local $h^\ast$-polynomials of $s$-lecture hall simplices.  
Consequently, we were able to answer some open questions pertaining to the real-rootedness of certain local $h$-polynomials for flag geometric subdivisions of a simplex.  
In the context of convex lattice polytopes, the real-rootedness of the $s$-derangement polynomials allowed us to recover that all $s$-lecture hall order polytopes have a box unimodal triangulation, that they all have unimodal local $h^\ast$-polynomials, and the $h^\ast$-polynomials of reflexive $s$-lecture hall order polytopes are always unimodal.  
Additionally, we could further use these results to provide a partial answer to a conjecture on the unimodality of the $h^\ast$-polynomials of a family of $s$-lecture hall order polytopes.  

In this way, it seems that the $s$-derangement polynomials provide a context in which various questions pertaining to local $h$-polynomials of subdivisions of simplicial complexes and $h^\ast$-polynomials of lattice polytopes can be simultaneously addressed.  
To better understand possible future applications of these methods, it could be useful to further analyze the family of $s$-derangement polynomials.  
In particular, considering the converse to Question~\ref{quest: local relations}, it would be interesting to know if every $s$-derangement polynomial can be realized as the local $h$-polynomial of a subdivision of a simplex.   
By \cite[Example 7.19]{KS16}, one way this could be affirmed is to show that all $s$-lecture hall simplices admit a regular and unimodular triangulation. 
Such a triangulation has recently been identified for some $s$-lecture hall simplices \cite{BBKSZb, BS20, HOT16}, but not all.  
Since, here, we have shown that all $s$-derangement polynomials are $\gamma$-nonnegative, a deeper understanding of the triangulations of $s$-lecture hall simplices could thereby offer new insights on the theory of local $h$-polynomials and questions of $\gamma$-nonnegativity for flag homology spheres and their subdivisions.  
Furthermore, it would also be of general interest to see a strictly combinatorial proof of the $\gamma$-nonnegativity of all $s$-derangement polynomials.

\smallskip

\noindent
{\bf Acknowledgements}. 
Liam Solus was supported by an NSF Mathematical Sciences Postdoctoral Research Fellowship (DMS - 1606407), the Wallenberg Autonomous Systems and Software Program (WASP), and Vetenskapsr\aa{}det. 
This paper is based on the master's thesis of the first author \cite{G18}, which was supervised by the second.  The authors would like to thank an anonymous referee for their thoughtful comments on this paper.

\end{document}